\documentclass[a4paper, 12pt]{article}
\usepackage[T2A]{fontenc}
\usepackage[english,russian]{babel}
\usepackage{amsmath}
\usepackage{amsfonts,amssymb,mathrsfs,amscd}
\usepackage{bm}
\newtheorem{theorem}{Theorem}[subsection]
\newtheorem{lemma}[theorem]{Lemma} 
\newtheorem{propos}[theorem]{Proposition} 
\newtheorem{definition}{Definition}

\newtheorem{remark}{Remark}
\newtheorem{proof}{Proof}

\newcommand{\Ker}{\operatorname{Ker}}

\newcommand{\Kappa}{\operatorname{\mathcal K}}
\newcommand{\sgn}{\operatorname{sign}}
\newcommand{\loc}{\operatorname{loc}}

\newcommand{\e}{\mathfrak e}

\newcommand{\mx}{\operatorname{\mathfrak M}}

\newcommand{\cmx}{\operatorname{\mathfrak C}}
\newcommand{\mn}{\operatorname{\mathfrak m}}
\newcommand{\cmn}{\operatorname{\mathfrak c}}
\newcommand{\m}{\mathfrak m}
\newcommand{\n}{\mathfrak n}
\newcommand{\M}{\mathfrak M}
\newcommand{\s}{\bm s}

\newcommand{\card}{\operatorname{card}}

\newcommand{\supp}{\operatorname{supp}}
\newcommand{\supvrai}{\operatornamewithlimits{sup\,vrai}}
\newcommand{\N}{\mathbb N}
\newcommand{\Z}{\mathbb Z}
\newcommand{\R}{\mathbb R}
\newcommand{\Nu}{\mathcal N}

\newcommand{\D}{\mathcal D}

\renewcommand{\L}{\mathcal L}

\allowdisplaybreaks

\begin{document}

\author{ S. N. Kudryavtsev }
\title{Approximation of partial differential operators on Nikol'skii -- Besov classes of mixed smoothness in domains of a certain kind}
\date{}
\maketitle
\begin{abstract}
In this paper, we consider Nikol'skii and Besov spaces with norms in which, instead of the mixed moduli of continuity of known orders of certain mixed derivatives of
 functions, the "$L_p$-averaged" mixed moduli of continuity of functions of the corresponding orders are used. In the problem of S.B. Stechkin, upper and lower estimates for the best accuracy of approximation of partial differential operators on classes of such functions defined in domains of a certain kind are established. These
 estimates are not weaker, and in some cases are stronger than the corresponding estimates obtained earlier by the author in the considered problem for these classes of 
functions on the cube $ I^d. $ At the same time, the class of Nikol'skii -- Besov spaces of mixed smoothness for which the mentioned estimates in the considered problem
 are obtained is significantly expanded.
\end{abstract}


\section*{Introduction}

This paper deals with the problem of S.B. Stechkin on the approximation in the norm of the space $ L_q $ of the partial differential operator $ \D^{\lambda}$ by bounded operators acting from the space $L_s $ into the space $ L_q, $ whose norm does not exceed $ \rho $, on the Nikol'skii $ (\mathcal S_p^\alpha \mathcal H)^\prime $ and Besov $ (\mathcal S_{p,\theta}^\alpha \mathcal B)^\prime $ classes. Under the corresponding conditions on $ \lambda \in \Z_+^d, d \in \N $, upper and lower estimates of the best accuracy of approximation in the space $ L_q(D), 1 \le q \le \infty, $ of the operator $ \D^\lambda $ by operators acting from $ L_s(D) $ into $ L_q(D) $, whose norm does not exceed $ \rho, $ on the Nikol'skii $ (\mathcal S_p^\alpha \mathcal H)^\prime(D) $ and Besov $ (\mathcal S_{p,\theta}^\alpha \mathcal B)^\prime(D), \alpha \in \R_+^d, 1 \le p < \infty, 1 \le \theta < \infty, $ classes defined in domains $ D \subset \R^d, $ of a certain kind, are established.

The present work continues the research carried out by the author in [1] -- [6] for the mentioned problem with respect to various classes of functions of finite smoothness. Note that the tools used in [1] -- [6] to obtain upper estimates of the above quantity for the classes of functions considered in these works are not applicable to the classes studied in this article. In addition, the schemes used in [1] -- [6] to obtain metric relations used to derive upper estimates of the accuracy of approximation of derivatives of functions from the classes studied in [1] -- [6] are unsuitable for the classes studied in the present work. The tools and schemes for deriving estimates proposed below made it possible to find, on almost the entire set of parameter values defining the objects under consideration, the weak asymptotics of the above quantity for the Nikol'skii and Besov classes of mixed smoothness of functions defined in domains of a certain structure. Thus, the set of function classes of Nikol'skii -- Besov mixed smoothness for which the weak asymptotics of the behavior, depending on $\rho $, of the best accuracy of approximation in the considered problem is established is significantly expanded.
At the same time, for $ \cmn(\tau^{-1} \gamma) \ne 1, \theta \ne \infty $, a stronger (see (2.2.11)) upper estimate than that previously established in [4] for the quantity considered here for such function classes on the cube $ I^d $ is obtained.

The work consists of an introduction and three sections, the first of which contains preliminary information, the second establishes an upper estimate, and the third establishes a lower estimate of the quantity under study.

\section{Preliminary information and auxiliary statements}

\subsection{} In this subsection, notations related to the functional classes and spaces defined in domains from $ \R^d $ ($d \in \N$ everywhere below), considered in this work, are introduced, and some facts necessary for the subsequent exposition are provided.

Let $ \Z_+^d $ denote the set
$$
\Z_+^d = \{\lambda = (\lambda_1, \ldots, \lambda_d) \in \Z^d:
\lambda_j \ge0, j=1, \ldots, d\}.
$$
For $ l \in \Z_+^d $, we also denote by $ \Z_+^d(l) $ the set
\begin{equation*} \tag{1.1.1}
\Z_+^d(l) = \{ \lambda  \in \Z_+^d: \lambda_j \le l_j, j=1, \ldots, d\}.
\end{equation*}

For $ l \in \Z_+^d $, let $ \mathcal P^{l} $ denote the space of real polynomials consisting of all functions $ f: \R^d \mapsto \R $ of the form
$$
f(x) = \sum_{\lambda \in \Z_+^d(l)} a_{\lambda} \cdot
x^{\lambda}, x \in \R^d,
$$
where $ a_{\lambda} \in \R,  x^{\lambda} =
x_1^{\lambda_1} \ldots x_d^{\lambda_d}, \lambda \in \Z_+^d(l). $

In $ \R^d $, we fix the norm
\begin{equation*}
\|x\| = \max_{j =1, \ldots,d} |x_j|.
\end{equation*}

For a set $ A $ from a topological space $ T $, $ \overline A $ denotes the closure of the set $ A. $ 

Let $ \chi_A $ denote the characteristic function of the set $ A \subset \R^d. $

For a Lebesgue measurable set $ D \subset \R^d $ and $ 1 \le p \le \infty $, $ L_p(D) $ denotes, as usual, the space of all real measurable functions $ f $ on $ D $ for which the norm
$$
\|f\|_{L_p(D)} = \begin{cases}
(\int_D |f(x)|^p dx)^{1/p}, 1 \le p < \infty; \\
\supvrai_{x \in D}|f(x)|, p = \infty,
\end{cases}
$$
is defined; here, as usual, functions that coincide almost everywhere on $ D $ are identified.

For $ x,y \in \R^d $, put $ xy = x \cdot y =
(x_1 y_1, \ldots, x_d y_d), $
and for $ x \in \R^d $ and $ A \subset \R^d $ define
$$
x A = x \cdot A = \{xy: y \in A\}.
$$

For $ x,y \in \R^d $, denote
$$
(x,y) = \sum_{j =1}^d x_j y_j.
$$

For $ x \in \R^d: x_j \ne 0, $ for $ j=1,\ldots,d,$ put
$ x^{-1} = (x_1^{-1},\ldots,x_d^{-1}). $

For $ x,y \in \R^d $, we write $ x \le y (x < y) $ if for each $ j=1,\ldots,d $ the inequality $ x_j \le y_j (x_j < y_j) $ holds.

For $ x \in \R^d $, put
$$
x_+ = ((x_1)_+, \ldots, (x_d)_+),
$$
where $ t_+ = \frac{1} {2} (t +|t|), t \in \R. $

Denote by $ \R_+^d $ the set of $ x \in \R^d $ such that $ x_j >0 $ for $ j=1,\ldots,d, $ and for $ a \in \R_+^d, x \in \R^d $ put $ a^x = a_1^{x_1} \ldots a_d^{x_d}.$

Define the sets
$$
I^d = \{x \in \R^d: 0 < x_j < 1,j=1,\ldots,d\},
$$
$$
\overline I^d = \{x \in \R^d: 0 \le x_j \le 1,j=1,\ldots,d\},
$$
$$
B^d = \{x \in \R^d: -1 \le x_j \le 1,j=1,\ldots,d\}.
$$

Let $ \e $ denote the vector in $ \R^d $ given by $ \e = (1,\ldots,1). $

For $ \lambda \in \Z_+^d $, let $ \D^\lambda $ denote the differential operator $ \D^\lambda =
\frac{\partial^{|\lambda|}} {\partial x_1^{\lambda_1} \ldots \partial x_d^{\lambda_d}}, $
where $ |\lambda| = \sum_{j=1}^d \lambda_j. $

Next, recall that for an open set $ D \subset \R^d $ and a vector $ h \in \R^d $, $ D_h $ denotes the set
$$
D_h = \{x \in D: x +th \in D \ \forall t \in \overline I\}.
$$

For a function $ f $ defined on an open set $ D \subset \R^d $ and a vector $ h \in \R^d $, define its difference $ \Delta_h f $ with step $ h $ on $ D_h $ by setting
$$
(\Delta_h f)(x) = f(x +h) -f(x), x \in D_h,
$$
and for $ l \in \N: l \ge 2, $ define the $l$-th difference $ \Delta_h^l f $ of $ f $ with step $ h $ on $ D_{lh} $ by the equality
$$
(\Delta_h^l f)(x) = (\Delta_h (\Delta_h^{l-1} f))(x),
x \in D_{lh},
$$
also set $ \Delta_h^0 f = f. $

As is known, the equality holds
$$
(\Delta_h^l f)(\cdot) = \sum_{k=0}^l C_l^k (-1)^{l-k}
f(\cdot +kh),
C_l^k =\frac{l!} {k! (l-k)!}.
$$

For $ j=1,\ldots,d$, let $ e_j $ denote the vector $ e_j = (0,\ldots,0,1_j,0,\ldots,0).$

As shown in [7], the following lemma holds.

\begin{lemma} \label{l1.1.1}

   For $ l \in \Z_+^d $
for any $ \delta \in \R_+^d $ and $ x^0 \in \R^d $ for $ Q = x^0 +\delta I^d $
there exists a unique linear operator
$ P_{\delta, x^0}^{l}: L_1(Q) \mapsto
\mathcal P^{l}, $
possessing the following properties:
\begin{enumerate}
\item for $ f \in \mathcal P^{l} $ the equality holds
\begin{equation*}
P_{\delta, x^0}^{l}(f \mid_Q) = f,
  \end{equation*}

\item \begin{equation*}
\Ker P_{\delta,x^0}^{l} = \biggl\{\,f \in L_1(Q):
\int \limits_{Q} f(x) g(x) \,dx =0\ \forall g \in \mathcal P^{l}\,\biggr\},
\end{equation*}

and there exist constants $ c_1(l) >0 $ and $ c_2(l) >0 $ such that

\item for $ 1 \le p \le \infty $ and $ f \in L_p(Q) $ the inequality holds
  \begin{equation*}
\|P_{\delta, x^0}^{l} f \|_{L_p(Q)} \le c_1
\|f\|_{L_p(Q)},
  \end{equation*}

\item for $ 1 \le p < \infty $ and $ f \in L_p(Q) $ the inequality holds
  \begin{equation*} \tag{1.1.2}
  \| f -P_{\delta, x^0}^{l} f \|_{L_p(Q)} \le c_2
\sum_{j=1}^d \delta_j^{-1/p} \biggl(\int_{\delta_j B^1} \int_{Q_{(l_j +1) \xi e_j}}
|\Delta_{\xi e_j}^{l_j +1} f(x)|^p dx d\xi\biggr)^{1/p}.
\end{equation*}
\end{enumerate}
\end{lemma}

Now we define the function spaces and classes considered in this paper (cf. [8], [9]). But first we introduce some notation.

For $ x \in \R^d $, let $\s(x) $ denote the set $ \s(x) = \{j =1,\ldots,d: x_j \ne 0\}, $ and for a set $ J \subset \{1,\ldots,d\} $ let $ \chi_J $ denote the vector from $ \R^d $ with components
$$
(\chi_J)_j = \begin{cases} 1, & \text{ for } j \in J; \\
0, & \text{ for } j \in (\{1,\ldots,d\} \setminus J).
\end{cases}
$$

For $ x \in \R^d $ and $ J = \{j_1,\ldots,j_k\}
\subset \N: 1 \le j_1 < j_2 < \ldots < j_k \le d, $ let $ x^J $
denote the vector $ x^J = (x_{j_1},\ldots,x_{j_k}) \in \R^k, $ and for a set $ A \subset \R^d $ put $ A^J = \{x^J: x \in A\}. $

For an open set $ D \subset \R^d $ and vectors $ h \in \R^d $ and $ l \in \Z_+^d $, let $ D_h^l $ denote the set
\begin{multline*}
D_h^l = (\ldots (D_{l_d h_d e_d})_{l_{d-1} h_{d-1} e_{d-1}}
\ldots)_{l_1 h_1 e_1} = \{ x \in D: x +tlh \in D \ \forall t \in
\overline I^d\} = \\ \{ x \in D: (x +\sum_{j \in \s(l)} t_j l_j h_j e_j) \in D \
\forall t^{\s(l)} \in (\overline I^d)^{\s(l)} \}.
\end{multline*}

Let $ D $ be an open set in $ \R^d $ and $ 1 \le p \le \infty. $
Then for $ f \in L_p(D), h \in \R^d $ and $ l \in \Z_+^d $, define on $ D_h^l $ the mixed difference of $ f $ of order $ l $ corresponding to the vector $ h $ by the equality
\begin{multline*}
(\Delta_h^l f)(x) = \biggl(\biggl(\prod_{j=1}^d
\Delta_{h_j e_j}^{l_j}\biggr) f\biggr)(x)
= \biggl(\biggl(\prod_{j \in \s(l)}
\Delta_{h_j e_j}^{l_j}\biggr) f\biggr)(x) = \\
\sum_{k \in \Z_+^d(l)} (-\e)^{l-k} C_l^k f(x+kh), x \in D_h^l,
\end{multline*}
where $ C_l^k = \prod_{j=1}^d C_{l_j}^{k_j}. $

Keeping in mind that for $ f \in L_p(D), l \in \Z_+^d $ and vectors $ h,h^\prime \in \R^d: h^{\s(l)} =
(h^\prime)^{\s(l)}, $ the relation holds
$$
\| \Delta_h^l f\|_{L_p(D_h^l)} = \| \Delta_{h^\prime}^l
f\|_{L_p(D_{h^\prime}^l)}, 1 \le p \le \infty,
$$
for $ 1 \le p \le \infty $ and $ f \in L_p(D) $ define the mixed modulus of continuity in $ L_p(D) $ of order $ l \in \Z_+^d $ by the equality
$$
\Omega^l (f,t^{\s(l)})_{L_p(D)} = \supvrai_{ \{ h \in \R^d: h^{\s(l)} \in t^{\s(l)} (B^d)^{\s(l)} \}}
\| \Delta_h^l f\|_{L_p(D_h^l)},
t^{\s(l)} \in (\R_+^d)^{\s(l)}.
$$
Furthermore, under the same conditions, introduce for the function $ f $ the "averaged" mixed modulus of continuity in $ L_p(D) $ of order $ l, $ setting
\begin{multline*}
\Omega^{\prime l} (f, t^{\s(l)})_{L_p(D)} =
\begin{cases}
\biggl((2 t^{\s(l)})^{-\e^{\s(l)}}
\int_{ t^{\s(l)} (B^d)^{\s(l)}}
\| \Delta_\xi^l f\|_{L_p(D_\xi^l)}^p
d \xi^{\s(l)}\biggr)^{1 /p} = \\
\biggl((2 t^{\s(l)})^{-\e^{\s(l)}}
\int_{ (t B^d)^{\s(l)}} \int_{D_\xi^{l \chi_{\s(l)}}}
| \Delta_\xi^{l \chi_{\s(l)}} f(x)|^p dx
d \xi^{\s(l)}\biggr)^{1 /p}, p \ne \infty; \\
\Omega^l (f,t^{\s(l)})_{L_p(D)}, p = \infty,
\end{cases} \\
t^{\s(l)} \in (\R_+^d)^{\s(l)}.
\end{multline*}

From the definitions given it is clear that
\begin{multline*} \tag{1.1.3}
\Omega^{\prime l} (f, t^{\s(l)})_{L_p(D)} \le
\Omega^l (f, t^{\s(l)})_{L_p(D)}, \\
t^{\s(l)} \in (\R_+^d)^{\s(l)}, 
f \in L_p(D), 1 \le p \le \infty, l \in \Z_+^d, \\
D \text{ -- arbitrary open set in } \R^d.
\end{multline*}

Now let $ \alpha \in \R_+^d, 1 \le p \le \infty $ and $ D $ be a domain in $ \R^d. $ Then define the vector $ l = l(\alpha) \in \N^d, $ by setting $ l_j = \min \{m \in \N: \alpha_j < m \}, j =1,\ldots,d, $ and denote by $ (S_p^\alpha H)^\prime(D) ((\mathcal S_p^\alpha \mathcal H)^\prime(D)) $
the set of all functions $ f \in L_p(D) $ such that for any nonempty set $ J \subset \{1,\ldots,d\} $ the inequality holds
$$
\sup_{t^J \in (\R_+^d)^J}
(t^J)^{-\alpha^J} \Omega^{\prime l \chi_J}(f,
t^J)_{L_p(D)}
= \sup_{t^J \in (\R_+^d)^J} (\prod_{j \in J}
t_j^{-\alpha_j})
\Omega^{\prime l \chi_J}(f, t^{\s(l\chi_J)})_{L_p(D)} < \infty (\le 1).
$$

In the space $ (S_p^\alpha H)^\prime(D) $, the norm is defined by
$$
\| f \|_{(S_p^\alpha H)^\prime(D)} = 
\max \biggl(\| f \|_{L_p(D)}, \max_{J \subset \{1, \ldots, d\}:
J \ne \emptyset} \sup_{t^J \in (\R_+^d)^J}
(t^J)^{-\alpha^J}
\Omega^{\prime l \chi_J}(f, t^J)_{L_p(D)}\biggr), 
$$
$f \in (S_p^\alpha H)^\prime(D)$.

Let $ \alpha,p,D $ and $ l = l(\alpha) $ be the same as above, and $ \theta \in \R: 1 \le \theta < \infty. $
Then denote by $ (S_{p,\theta}^\alpha B)^\prime(D) ((\mathcal S_{p,\theta}^\alpha \mathcal B)^\prime(D)) $ the set of all functions $ f \in L_p(D) $ which for any nonempty set $ J \subset \{1,\ldots,d\} $ satisfy the condition
\begin{multline*}
\biggl(\int_{(\R_+^d)^J} (t^J)^{-\e^J -\theta \alpha^J}
(\Omega^{\prime l \chi_J}(f, t^J)_{L_p(D)})^\theta dt^J\biggr)^{1/\theta} =\\
\biggl(\int_{(\R_+^d)^J} (\prod_{j \in J}
t_j^{-1 -\theta \alpha_j})
(\Omega^{\prime l \chi_J}(f,
t^{\s(l \chi_J)})_{L_p(D)}^\theta
\prod_{j \in J} dt_j\biggr)^{1/\theta} < \infty (\le 1).
\end{multline*}

In the space $ (S_{p,\theta}^\alpha B)^\prime(D) $ the norm is defined by
\begin{multline*}
\| f \|_{(S_{p,\theta}^\alpha B)^\prime(D)} = \\
\max \biggl(\| f\|_{L_p(D)},
\max_{J \subset \{1, \ldots, d\}: J \ne \emptyset}
\biggl(\int_{(\R_+^d)^J} (t^J)^{-\e^J -\theta \alpha^J}
(\Omega^{\prime l \chi_J}(f, t^J)_{L_p(D)})^\theta dt^J\biggr)^{1/\theta}\biggr), \\
 \ f \in (S_{p,\theta}^\alpha B)^\prime(D).
\end{multline*}

For $ \theta = \infty $ put $ (S_{p,\infty}^\alpha B)^\prime(D) =
(S_p^\alpha H)^\prime(D), (\mathcal S_{p,\infty}^\alpha \mathcal B)^\prime(D) =
(\mathcal S_p^\alpha \mathcal H)^\prime(D). $

Taking into account that for $ f \in (S_{p,\theta}^\alpha
B)^{\prime}(D), t^J \in (\R_+^d)^J (J \subset
\{1, \ldots, d\}: J \ne \emptyset) $
the inequality holds (see [10])
\begin{multline*}
(t^J)^{-\alpha^J} \Omega^{\prime l \chi_J}(f,
t^J)_{L_p(D)} \le \\
(\prod_{j \in J} 2^{\alpha_j +1/\theta +1 /p})
\biggl(\int_{(\R_+^d)^J} (\tau^J)^{-\e^J -\theta \alpha^J} (\Omega^{\prime l \chi_J}(f,
\tau^J)_{L_p(D)})^\theta d\tau^J\biggr)^{1/\theta},
\end{multline*}
we conclude that
\begin{equation*} \tag{1.1.4}
(\mathcal S_{p, \theta}^\alpha \mathcal B)^\prime(D) \subset c_3(\alpha)
(\mathcal S_p^\alpha \mathcal H)^\prime(D),
\end{equation*}
where $ c_3(\alpha) = \prod_{j=1}^d 2^{2+\alpha_j}. $

Let $ \alpha,p,D $ and $ l = l(\alpha) $ be the same as above, and $ \theta \in \R: 1 \le \theta < \infty. $ Then denote by $ (S_{p,\theta}^\alpha B)^0(D)
((\mathcal S_{p,\theta}^\alpha \mathcal B)^0(D)) $
the set of all functions $ f \in L_p(D) $ which for any nonempty set $ J \subset \{1,\ldots,d\} $ satisfy the condition
\begin{multline*}
\biggl(\int_{(\R_+^d)^J} (t^J)^{-\e^J -\theta \alpha^J}
(\Omega^{l \chi_J}(f, t^J)_{L_p(D)})^\theta
dt^J\biggr)^{1/\theta} = \\
\biggl(\int_{(\R_+^d)^J} (\prod_{j \in J}
t_j^{-1 -\theta \alpha_j})
(\Omega^{l \chi_J}(f, t^{\s(l \chi_J)})_{L_p(D)})^\theta \prod_{j \in J} dt_j\biggr)^{1/\theta} < \infty (\le 1),
\end{multline*}
and by $ (S_p^\alpha H)^0(D) ((\mathcal S_p^\alpha \mathcal H)^0(D)) $ --
the set of all functions $ f \in L_p(D) $ such that for any nonempty set $ J \subset \{1,\ldots,d\} $ the inequality holds
$$
\sup_{t^J \in (\R_+^d)^J}
(t^J)^{-\alpha^J} \Omega^{l \chi_J}(f,t^J)_{L_p(D)}
= \sup_{t^J \in (\R_+^d)^J} (\prod_{j \in J}
t_j^{-\alpha_j})
\Omega^{l \chi_J}(f, t^{\s(l\chi_J)})_{L_p(D)} < \infty (\le 1).
$$
For $ \theta = \infty $ put $ (S_{p,\infty}^\alpha B)^0(D) =
(S_p^\alpha H)^0(D), (\mathcal S_{p,\infty}^\alpha \mathcal B)^0(D) =
(\mathcal S_p^\alpha \mathcal H)^0(D). $

The norm in the space $ (S_{p,\theta}^\alpha B)^0(D) $ is defined by replacing in the definition of the norm $ \|
\cdot\|_{(S_{p,\theta}^\alpha B)^\prime(D)} $ the quantity $ \Omega^{\prime l \chi_J}(f,t^J)_{L_p(D)} $ by $ \Omega^{l
\chi_J}(f,t^J)_{L_p(D)}. $

From (1.1.3) it follows that
\begin{multline*} \tag{1.1.5}
\| f\|_{(S_{p, \theta}^\alpha B)^\prime(D)} \le
\| f\|_{(S_{p, \theta}^\alpha B)^{0}(D)}, \\
\ f \in (S_{p, \theta}^\alpha B)^{0}(D),
\alpha \in \R_+^d, 1 \le p \le \infty, 1 \le \theta \le \infty, \\
D \text{ -- arbitrary domain in } \R^d.
\end{multline*}

Denote by $ C^\infty(D) $ the space of infinitely differentiable functions on an open set $ D \subset \R^d, $ and by $ C_0^\infty(D) $ -- the space of functions $ f \in C^\infty(\R^d) $ each of which has compact support $ \supp f \subset D. $
Furthermore, let $ L_1^{\loc}(D) $ denote the space of real locally integrable functions on an open set $ D \subset \R^d, $ i.e., the space of real functions on $ D $ that are summable on any compact set lying in $ D. $

In conclusion of this subsection, we introduce a few more notations.

For a Banach space $ X $ (over $ \R$), denote $ B(X) = \{x \in X: \|x\|_X \le 1\}. $

For Banach spaces $ X,Y $, let $ \mathcal B(X,Y) $ denote the Banach space consisting of continuous linear operators $ T: X \mapsto Y, $ with norm
$$
\|T\|_{\mathcal B(X,Y)} = \sup_{x \in B(X)} \|Tx\|_Y.
$$

\subsection{}
This subsection contains information on multiple series that will be used later.

For $ y \in \R^d $, set
$$
\mn(y) = \min_{j=1,\ldots,d} y_j
$$
and for a Banach space $ X, $ a vector $ x \in X $ and a family $ \{x_\kappa \in X, \kappa \in \Z_+^d\} $, we write $ x = \lim_{ \mn(\kappa) \to \infty} x_\kappa, $
if for any $ \epsilon >0 $
there exists $ n_0 \in \N $ such that for any $ \kappa \in \Z_+^d $ for which $ \mn(\kappa) > n_0, $ the inequality $ \|x -x_\kappa\|_X < \epsilon $ holds.

Let $ X $ be a Banach space and $ \{ x_\kappa \in X, \kappa \in \Z_+^d\} $ be a family of vectors.
Then by the sum of the series $ \sum_{\kappa \in \Z_+^d} x_\kappa $
we mean a vector $ x \in X $ for which the equality holds
$ x = \lim_{\mn(k) \to \infty} \sum_{\kappa \in
\Z_+^d(k)} x_\kappa. $

Let $ \Upsilon^d $ denote the set
$$
\Upsilon^d = \{ \epsilon \in \Z^d: \epsilon_j \in
\{0,1\}, j=1,\ldots,d\}.
$$

The following lemma holds.

\begin{lemma} \label{l1.2.1}

Let $ X $ be a Banach space, and let a vector $ x \in X $ and a family $ \{x_\kappa \in X, \kappa \in \Z_+^d\} $ be such that $ x = \lim_{ \mn(\kappa) \to \infty} x_\kappa, $ Then for the family $ \{ \mathcal X_\kappa \in X, \kappa \in \Z_+^d \} $ whose elements are defined by the equality
$$
\mathcal X_\kappa = \sum_{\epsilon \in \Upsilon^d:
\s(\epsilon) \subset \s(\kappa)} (-\e)^\epsilon
x_{\kappa -\epsilon}, \kappa \in \Z_+^d,
$$
the equality holds
$$
x = \sum_{\kappa \in \Z_+^d} \mathcal X_\kappa.
$$
\end{lemma}

The lemma is a consequence of the fact that for $ k \in \Z_+^d $
the equality holds
\begin{equation*}
\sum_{\kappa \in \Z_+^d(k)}
\mathcal X_\kappa = x_k \text{ (see [7])}.
\end{equation*}

\begin{remark} \label{z1}

It is easy to see that for any family of numbers
$ \{x_\kappa \in \R: x_\kappa \ge 0, \kappa \in
\Z_+^d\}, $ if the series $ \sum_{\kappa \in \Z_+^d} x_\kappa $ converges, i.e., there exists a limit
$ \lim_{\mn(k) \to \infty} \sum_{\kappa \in \Z_+^d(k)} x_\kappa, $
which is equivalent to the relation
$ \sup_{k \in \Z_+^d} \sum_{\kappa \in \Z_+^d(k)}
x_\kappa < \infty, $

then for any sequence of subsets
$ \{Z_n \subset \Z_+^d, n \in \Z_+\}, $
 such that $ \card Z_n < \infty,
Z_n \subset Z_{n+1},  n \in \Z_+, $
and $ \cup_{ n \in \Z_+} Z_n = \Z_+^d, $
the equality holds
$ \sum_{\kappa \in \Z_+^d} x_\kappa =
\lim_{ n \to \infty} \sum_{\kappa \in Z_n} x_\kappa =
\sup_{k \in \Z_+^d} \sum_{\kappa \in \Z_+^d(k)}
x_\kappa. $
From this it is easy to understand that if for a family of vectors
$ \{x_\kappa \in X, \kappa \in \Z_+^d\}  $ of a Banach space $ X $
the series $ \sum_{\kappa \in \Z_+^d} \| x_\kappa \|_X $ converges,
then for any sequence of subsets
$ \{Z_n \subset \Z_+^d, n \in \Z_+\}, $
 such that $ \card Z_n < \infty,
Z_n \subset Z_{n+1},  n \in \Z_+, $
and $ \cup_{ n \in \Z_+} Z_n = \Z_+^d, $
the equality holds in $ X $
$ \sum_{\kappa \in \Z_+^d} x_\kappa =
\lim_{ n \to \infty} \sum_{\kappa \in Z_n} x_\kappa. $
\end{remark}

For $ x \in \R^d $, denote
\begin{eqnarray*}
\mx(x)  &=& \max_{j=1,\ldots,d} x_j, \\
\cmx(x) &=& \card \{j \in \{1,\ldots,d\}: x_j =
\mx(x)\}, \\
\cmn(x) &=& \card \{j \in \{1,\ldots,d\}: x_j =
\mn(x)\}.
\end{eqnarray*}

\begin{lemma} \label{l1.2.2}

Let $ \alpha, \beta \in \R_+^d. $ Then there exist constants $ c_1(\alpha,\beta) >0 $ and $ c_2(\alpha,\beta) >0 $ such that for $ r \in \N $ the inequality holds
\begin{equation*} \tag{1.2.1}
c_1 2^{-\mn(\beta^{-1} \alpha) r} r^{\cmn(\beta^{-1} \alpha) -1}
\le \sum_{\kappa \in \Z_+^d: (\kappa, \beta) >
 r} 2^{-(\kappa, \alpha)} \le
c_2 2^{-\mn(\beta^{-1} \alpha) r} r^{\cmn(\beta^{-1} \alpha) -1}.
\end{equation*}
\end{lemma}

\begin{lemma} \label{l1.2.3}

Let $ \beta \in \R_+^d, \alpha \in \R^d $ and $ \mx(\beta^{-1} \alpha) >0. $ Then there exist constants $ c_3(\alpha,\beta)  >0 $ and $ c_4(\alpha,\beta) >0 $ such that
for $ r \in \N $ the inequalities hold
\begin{equation*} \tag{1.2.2}
c_3 2^{ \mx(\beta^{-1} \alpha) r} r^{ \cmx(\beta^{-1} \alpha) -1}
\le \sum_{ \kappa \in \Z_+^d: (\kappa, \beta) \le r}
2^{(\kappa, \alpha)} \le
c_4 2^{ \mx(\beta^{-1} \alpha) r} r^{ \cmx(\beta^{-1} \alpha) -1}.
\end{equation*}
\end{lemma}
The proof of relations (1.2.1) and (1.2.2) is given in [11].

\subsection{}
In this subsection, we introduce spaces of piecewise-polynomial functions and operators on them, which are used to construct approximation tools for functions from the spaces under study.

Consider a system of partitions of unity on open sets, with the help of which approximation tools for functions from the spaces under study are constructed. For this, denote by $ \psi^{0} $ the characteristic function of the interval $ I, $ i.e., the function defined by
$$
\psi^{0}(x) = \begin{cases} 1, & \text{ for } x \in I; \\
0, & \text{ for } x \in \R \setminus I.
\end{cases}
$$
For $ m \in \N $, set
$$
\psi^{m}(x) = \int_I \psi^{m -1}(x -y) dy \ (\text{see } [12]),
$$
and for $ d \ge 2, m \in \Z_+^d $ define
$$
\psi^{m}(x) = \prod_{j=1}^d \psi^{m_j}(x_j), x =
(x_1,\ldots,x_d) \in \R^d.
$$

For $ m,n \in \Z^d: m \le n, $ denote
\begin{equation*} \tag{1.3.1}
\Nu_{m,n} = \{ \nu \in \Z^d: m \le \nu \le n \} =
\prod_{j=1}^d \Nu_{m_j,n_j}.
\end{equation*}

Based on the definitions, using induction, it is easy to verify the following properties of the functions $ \psi^{m}, m \in \Z_+^d. $

1) For $ m \in \Z_+^d $
$$
\sgn \psi^{m}(x) = \begin{cases} 1, \text{ for } x \in ((m +\e) I^d); \\
0, \text{ for } x \in \R^d \setminus ((m +\e) I^d),
\end{cases}
$$

2) for $ m \in \Z_+^d $ and for each $ \lambda \in
\Z_+^d(m) $ the (generalized) derivative $ \D^\lambda \psi^{m} \in L_\infty(\R^d), $

3) for $ m \in \Z_+^d $ for almost all $ x \in \R^d $
the equality holds
$$
\sum_{\nu \in \Z^d} \psi^{m}(x -\nu) =1,
$$

4) for $ m \in \N $ for all $ x \in \R $ (for $ m =0 $ for almost all $ x \in \R $)
the equality holds
\begin{equation*} \tag{1.3.2}
\psi^{m}(x) = \sum_{\mu \in \Nu_{0, m+1}} a_{\mu}^m
\psi^{m}(2x -\mu),
\end{equation*}
where $ a_\mu^m = 2^{-m} C_{m+1}^\mu. $
Using the Newton expansion for $ (1+1)^{m+1} $ and $ (-1+1)^{m+1}, $ it is easy to verify that for $ m \in \Z_+ $ the equalities hold
\begin{equation*} \tag{1.3.3}
\sum_{\mu \in \Nu_{0,m +1} \cap (2 \Z)} a_\mu^m =1,
\sum_{\mu \in \Nu_{0,m +1} \cap (2 \Z +1)} a_\mu^m =1.
\end{equation*}

For $ t \in \R^d $, let $ 2^t $ denote the vector $ 2^t = (2^{t_1}, \ldots, 2^{t_d}). $

For $ m,\kappa \in \Z_+^d, \nu \in \Z^d $, denote
$$
g_{\kappa, \nu}^{m}(x) = \psi^{m}(2^\kappa x -\nu) =
\prod_{j =1}^d \psi^{m_j}( 2^{\kappa_j} x_j -\nu_j), x \in \R^d.
$$
From the first of the above properties of the functions $ \psi^{m} $ it follows that for $ m,\kappa \in \Z_+^d, \nu \in \Z^d $ the support
\begin{equation*} \tag{1.3.4}
\supp g_{\kappa,\nu}^{m} =
2^{-\kappa} \nu +2^{-\kappa} (m +\e) \overline I^d.
\end{equation*}
For $ \kappa \in \Z_+^d, \nu \in \Z^d $, denote
\begin{equation*} \tag{1.3.5}
Q_{\kappa, \nu} = 2^{-\kappa} \nu +2^{-\kappa} I^d,
\overline Q_{\kappa, \nu} = 2^{-\kappa} \nu +2^{-\kappa} \overline I^d.
\end{equation*}

Let us note some useful properties of the supports of the functions $ g_{\kappa,\nu}^{m}. $

For $ m,\kappa \in \Z_+^d $ for each $ \nu^\prime \in \Z^d $
the equality holds
\begin{equation*} \tag{1.3.6}
\{ \nu \in \Z^d: Q_{\kappa, \nu^\prime} \cap
\supp g_{\kappa, \nu}^{m} \ne \emptyset\} = \nu^\prime +\Nu_{-m,0}.
\end{equation*}

From property 3) of the functions $ \psi^{m} $ it follows that for $ m, \kappa \in \Z_+^d $
for any open set $ U \subset \R^d $
for almost all $ x \in U $ the equality holds
\begin{equation*} \tag{1.3.7}
\sum_{ \nu \in \Z^d: \supp g_{\kappa, \nu}^{m} \cap U \ne \emptyset} g_{\kappa, \nu}^{m}(x) =1.
\end{equation*}

Keeping in mind property 2) of the functions $ \psi^{m}, $ note that for $ m,\kappa \in \Z_+^d, \nu \in \Z^d, \lambda \in
\Z_+^d(m) \text{ (see (1.1.1))}$
the equality holds
\begin{multline*} \tag{1.3.8}
\| \D^\lambda g_{\kappa, \nu}^{m} \|_{L_\infty(\R^d)} =
2^{(\kappa, \lambda)}
\| \D^\lambda \psi^{m} \|_{L_\infty(\R^d)} =
c_1(m,\lambda) 2^{(\kappa, \lambda)}.
\end{multline*}
Let us introduce the following spaces of piecewise-polynomial functions.
For $ l \in \Z_+^d, m \in \N^d, $ an open set $ U \subset \R^d $ and $ \kappa \in \Z_+^d, $ setting
\begin{equation*} \tag{1.3.9}
N_\kappa = N_\kappa^{m,U} = \{\nu \in \Z^d:
\supp g_{\kappa, \nu}^{m}
\cap U \ne \emptyset\},
\end{equation*}
let $ \mathcal P_\kappa^{l,m,U} $ denote the linear space consisting of functions $ f: \R^d \mapsto \R $ for each of which there exists a set of polynomials
$ \{f_\nu \in \mathcal P^{l}, \nu \in N_\kappa^{m,U}\} $ such that
for $ x \in \R^d $ the equality holds
\begin{equation*} \tag{1.3.10}
f(x) = \sum_{\nu \in N_\kappa^{m,U}} f_\nu(x) g_{\kappa,\nu}^{m}(x).
\end{equation*}

\begin{remark} \label{z2}
In the case when $ \card N_\kappa < \infty, $ no explanation is needed as to what is meant by the sum $ \sum_{\nu \in N_\kappa} f_\nu(x) g_{\kappa,\nu}^m(x). $
In this case, it is easy to verify that for $ l \in \Z_+^d,
m \in \N^d, \kappa \in \Z_+^d $ and a bounded open set $ U \subset \R^d $
the mapping that assigns to each set of polynomials $ \{f_\nu \in \mathcal P^{l}, \nu \in N_\kappa \} $ the function $ f $ defined by
equality (1.3.10) is an isomorphism of the direct product of $ \card N_\kappa $ copies of the space $ \mathcal P^{l} $ onto the space $ \mathcal P_\kappa^{l,m,U}. $

If $ \card N_\kappa = \infty, $ then taking an arbitrary bijective mapping $ \N \ni s \mapsto \nu^s \in N_\kappa, $ and
taking into account that for $ x \in \R^d $, due to (1.3.4),
$ \card \{s \in \N: g_{\kappa,\nu^s}^m(x) \ne 0\} < \infty, $ set
\begin{equation*}
\sum_{\nu \in N_\kappa} f_\nu(x) g_{\kappa,\nu}^m(x) =
\sum_{s =1}^\infty f_{\nu^s}(x) g_{\kappa,\nu^s}^m(x) =
\sum_{s \in \N: g_{\kappa,\nu^s}^m(x) \ne 0}
f_{\nu^s}(x) g_{\kappa,\nu^s}^m(x),
\end{equation*}
and for any compact set $ K \subset \R^d $ for $ x \in K $ the sum
$$
\sum_{\nu \in N_\kappa} f_\nu(x) g_{\kappa,\nu}^m(x) =
\sum_{\nu \in N_\kappa: K \cap \supp g_{\kappa,\nu}^m \ne \emptyset} f_\nu(x) g_{\kappa,\nu}^m(x),
$$
and the sum on the right-hand side of the last equality, in view of (1.3.4), contains a finite number of terms (functions summable on
$ K $), thus, the inclusion holds
$ \mathcal P_\kappa^{l,m,U}
\subset L_1^{\loc}(\R^d). $
\end{remark}
Based on (1.3.2), in the same way as Lemma 1.2.1 from [5], the following lemma is established.

\begin{lemma}\label{l1.3.1}
Let $ l \in \Z_+^d, m \in \N^d, U $ be an open set in $ \R^d. $ Then for $ \kappa \in \Z_+^d, j =1,\ldots,d $
the linear operator $ H_\kappa^{j,l,m,U}:
\mathcal P_\kappa^{l,m,U} \mapsto
\mathcal P_{\kappa +e_j}^{l,m,U}, $
whose value on a function $ f \in
\mathcal P_\kappa^{l,m,U} $ given by
equality (1.3.10) is defined by the relation
\begin{multline*} \tag{1.3.11}
(H_\kappa^{j,l,m,U} f)(x) = \\
\sum_{\nu \in N_{\kappa +e_j}^{m,U}}
\biggl(\sum_{\substack{\nu^\prime \in N_\kappa^{m,U}, \mu_j \in \Nu_{0, m_j +1}: \\
 2 \nu^\prime_j +\mu_j = \nu_j, \nu^\prime_i = \nu_i, i = 1,\ldots,d, i \ne j }} a_{\mu_j}^{m_j}
f_{\nu^\prime}(x)\biggr) g_{\kappa +e_j,\nu}^{m}(x), x \in \R^d,
\end{multline*}
has the property that for $ f \in
\mathcal P_\kappa^{l,m,U} $ the equality holds
\begin{equation*}
(H_\kappa^{j,l,m,U} f) \mid_{U} = f \mid_{U}.
\end{equation*}
\end{lemma}
Note that if in the formulation of Lemma \ref{l1.3.1} the set $ U $ is not bounded, then the operator $ H_\kappa^{j,l,m,U} $ may be multivalued.

Below, the following objects will be needed.

For $ m \in \N^d, \epsilon \in \Upsilon^d, \nu \in
\Z^d $, denote by $ \M_{\epsilon}^m(\nu) $ the set of collections of numbers
\begin{multline*} \tag{1.3.12}
\M_{\epsilon}^m(\nu) = \{ \m^{\epsilon} = \{ \m_j \in \Nu_{0, m_j +1},
j \in \s(\epsilon)\}: \\
(\nu_j -\m_j) /2 \in \Z \ \forall j \in \s(\epsilon)\} = \\
\prod_{j \in \s(\epsilon)} \{ \m_j \in \Nu_{0, m_j +1}: (\nu_j -\m_j) /2 \in \Z\} =
\prod_{j \in \s(\epsilon)} \M_1^{m_j}(\nu_j),
\end{multline*}
and to each pair $ \nu \in \Z^d, \m^{\epsilon} \in
\M_{\epsilon}^m(\nu) $ associate $ \n_{\epsilon}(\nu,\m^{\epsilon}) \in \Z^d, $ by setting
\begin{equation*} \tag{1.3.13}
(\n_{\epsilon}(\nu,\m^{\epsilon}))_j = \begin{cases} (\nu_j -\m_j) /2,
j \in \s(\epsilon); \\
\nu_j, j \in \{1,\ldots,d\} \setminus \s(\epsilon).
\end{cases}
\end{equation*}

The following fact will be useful later. For $ m \in \N^d $ for
$ \nu \in \Z^d, \epsilon, \epsilon^\prime \in
\Upsilon^d: \s(\epsilon) \cap \s(\epsilon^\prime) = \emptyset, $
and any $ \m^{\epsilon} \in \M_{\epsilon}^m(\nu),
\m^{\epsilon^\prime} \in \M_{\epsilon^\prime}^m(\nu), \m^{\epsilon} * \m^{\epsilon^\prime} =
\m^{\epsilon +\epsilon^\prime} \in \M_{\epsilon +\epsilon^\prime}^m(\nu), $
with components defined by the relation
$$
(\m^{\epsilon} * \m^{\epsilon^\prime})_j = \begin{cases}
(\m^{\epsilon})_j, j \in \s(\epsilon); \\
(\m^{\epsilon^\prime})_j, j \in \s(\epsilon^\prime),
\end{cases}
$$
the equality holds
\begin{equation*} \tag{1.3.14}
\n_{\epsilon +\epsilon^\prime}(\nu, \m^{\epsilon} *
\m^{\epsilon^\prime}) =
\n_{\epsilon^\prime}(\n_{\epsilon}(\nu, \m^{\epsilon}), \m^{\epsilon^\prime}).
\end{equation*}

\begin{remark} \label{z3}
For $ m \in \N^d $, $ \kappa \in \Z_+^d,
\epsilon \in \Upsilon^d: \s(\epsilon) \subset
\s(\kappa), $ for an open set $ U \subset \R^d $ and $ \nu \in
N_{\kappa}^{m,U}, \m^{\epsilon} \in
\M_{\epsilon}^m(\nu) $ the inclusion holds (see [10])
\begin{equation*} \tag{1.3.15}
\n_{\epsilon}(\nu,\m^{\epsilon}) \in
N_{\kappa -\epsilon}^{m,U}.
\end{equation*}
\end{remark}
For the formulation of Lemma \ref{l1.3.2}, for $ j \in
\{1,\ldots,d\} $ denote by $ \eta^j: \R^d \times \R^d \mapsto \R^d $ the mapping defined by
$$
(\eta^j(\xi,x))_i = \begin{cases} \xi_i, i =1, \ldots, j; \\
x_i, i =j +1, \ldots, d,
\end{cases} \xi,x \in \R^d.
$$
Note also that Lemma \ref{l1.3.2} is established using Lemma \ref{l1.3.1}, and its proof simplified repeats the proof of Lemma 1.2.2 from [5].

\begin{lemma}\label{l1.3.2}
Let $ l \in \Z_+^d, m \in \N^d, U $ be an open set in $ \R^d $ and $ \kappa \in \Z_+^d, \epsilon \in
\Upsilon^d: \s(\epsilon) \subset \s(\kappa). $ Then the linear operator $ H_{\kappa, \kappa -\epsilon}^{l,m,U}:
\mathcal P_{\kappa -\epsilon}^{l,m,U} \mapsto
\mathcal P_\kappa^{l,m,U}, $
whose value for $ f \in \mathcal P_{\kappa -\epsilon}^{l,m,U} $ is defined by
\begin{equation*} \tag{1.3.16}
H_{\kappa, \kappa -\epsilon}^{l,m,U} f = \begin{cases} f, \text{ for }
\epsilon =0; \\
(\prod_{j \in \s(\epsilon)} H_{\eta^j(\kappa -\epsilon, \kappa)}^{j,l,m,U}) f,
\text{ for } \epsilon \ne 0, (\text{ see } (1.3.11)),
\end{cases}
\end{equation*}
possesses the following properties:

1) for $ f \in \mathcal P_{\kappa -\epsilon}^{l,m,U} $ the equality holds
\begin{equation*} \tag{1.3.17}
(H_{\kappa, \kappa -\epsilon}^{l,m,U} f) \mid_{U} =
f \mid_{U};
\end{equation*}

2) for $ f \in \mathcal P_{\kappa -\epsilon}^{l,m,U} $ of the form
\begin{equation*} \tag{1.3.18}
f = \sum_{\nu^\prime \in N_{\kappa -\epsilon}^{m,U}}
f_{\kappa -\epsilon, \nu^\prime} g_{\kappa -\epsilon, \nu^\prime}^{m},
\{f_{\kappa -\epsilon, \nu^\prime} \in \mathcal P^{l},
\nu^\prime \in N_{\kappa -\epsilon}^{m,U}\},
\end{equation*}
the representation holds
\begin{equation*} \tag{1.3.19}
H_{\kappa, \kappa -\epsilon}^{l,m,U} f = \sum_{\nu \in N_{\kappa}^{m,U}}
f_{\kappa,\nu} g_{\kappa,\nu}^{m},
\end{equation*}
where
\begin{equation*} \tag{1.3.20}
f_{\kappa,\nu} = \sum_{\m^{\epsilon} \in
\M_{\epsilon}^m(\nu)} A_{\m^{\epsilon}}^m
f_{\kappa -\epsilon, \n_{\epsilon}(\nu,\m^{\epsilon})}, \text{ (see (1.3.15))}
\end{equation*}
and
\begin{equation*} \tag{1.3.21}
A_{\m^{\epsilon}}^m = \prod_{i \in \s(\epsilon)}
a_{\m_i}^{m_i}, (\text{ see } (1.3.2)), \m^{\epsilon} \in \M_{\epsilon}^m(\nu),
\nu \in N_{\kappa}^{m,U}.
\end{equation*}
\end{lemma}
Let us also note a special case of Lemma 1.2.3 from [5] (see (1.3.3), (1.3.21)).

\begin{lemma} \label{l1.3.3}
For $ \nu \in \Z^d, \epsilon \in \Upsilon^d, m \in
\N^d $ the equality holds
\begin{equation*} \tag{1.3.22}
\sum_{\m^{\epsilon} \in \M_{\epsilon}^m(\nu)}
A_{\m^{\epsilon}}^m =1.
\end{equation*}
\end{lemma}

\section{Upper estimate of the best accuracy of approximation in $ L_q(D) $ of the operator $ \D^\lambda $
by bounded operators acting from $ L_s(D) $ into $ L_q(D) $, on classes $ (\mathcal S_{p,\theta}^\alpha \mathcal B)^\prime(D) $}

\subsection{} In this subsection, approximation tools on some open subsets of the domains of definition of functions from the spaces under consideration will be constructed, which satisfy metric relations useful for us.

For $ l, \kappa \in \Z_+^d, \nu \in \Z^d $, define the linear operator $ S_{\kappa,\nu}^{l}:
L_1(Q_{\kappa,\nu}) \mapsto \mathcal P^{l}, $
by setting $ S_{\kappa, \nu}^{l} = P_{\delta, x^0}^{l} $ for $ \delta = 2^{-\kappa}, x^0 = 2^{-\kappa} \nu $ (see Lemma \ref{l1.1.1} and (1.3.5)).

For an open set $ D \subset \R^d $, denote by $ \tilde L_1^{\loc}(D) $ the space of all functions $ f \in L_1^{\loc}(D) $ for which for any $ \kappa \in \Z_+^d, \nu \in \Z^d $ such that $ Q_{\kappa,\nu} \subset D, $ the inclusion $ f \mid_{Q_{\kappa,\nu}} \in L_1(Q_{\kappa,\nu}) $ holds.
Note that in the situation where $ Q_{\kappa,\nu} \subset D, $ for $ f \in \tilde L_1^{\loc}(D) $ we will write $ S_{\kappa, \nu}^{l} f $ instead of $ S_{\kappa, \nu}^{l}(f \mid_{Q_{\kappa, \nu}}). $

For a domain $ D \subset \R^d, $ its open subset $ U \subset D $ and $ \kappa \in \Z_+^d $ such that the set
\begin{equation*} \tag{2.1.1}
\{\nu^\prime \in \Z^d: Q_{\kappa,\nu^\prime} \subset D\} \ne \emptyset,
\end{equation*}
and $ m \in \N^d, $ fix some mapping
\begin{equation*} \tag{2.1.2}
\bm \nu_\kappa: N_\kappa^{m,U} \ni \nu \mapsto
\bm \nu_\kappa(\nu) \in \{\nu^\prime \in \Z^d:
Q_{\kappa,\nu^\prime} \subset D\} \text{ (see } (1.3.9)),
\end{equation*}
and for $ l \in \Z_+^d $ define the linear operator $ E_\kappa^{l,m,D,U,\bm \nu_\kappa}:
\tilde L_1^{\loc}(D) \mapsto
\mathcal P_\kappa^{l,m,U} $ (see Subsection 1.3.) by the equality
\begin{equation*} \tag{2.1.3}
E_\kappa^{l,m,D,U,\bm \nu_\kappa} f = \sum_{\nu \in N_\kappa^{m,U}}
(S_{\kappa, \bm \nu_\kappa(\nu)}^{l} f ) g_{\kappa, \nu}^{m},
f \in \tilde L_1^{\loc}(D).
\end{equation*}

\begin{remark}\label{z4}
If for a domain $ D \subset \R^d $ and $ \kappa^0
\in \Z_+^d $, (2.1.1) holds with $ \kappa^0 $ instead of $ \kappa, $
then for $ \kappa \in \Z_+^d $, (2.1.1) holds with $ \kappa^0 +\kappa $ instead of $ \kappa. $
\end{remark}
The following statement will be needed later.

\begin{lemma} \label{l2.1.1}
Let $ \lambda \in \Z_+^d, D $ be a domain in $ \R^d $ and a function $ f \in C^\infty(D), $ and $ g \in
L_1^{\loc}(D), $ and for each $ \mu \in \Z_+^d(\lambda) $ (see (1.1.1)) the generalized derivative $ \D^\mu g \in L_1^{\loc}(D). $ Then in the space of generalized functions in the domain $ D $ the relation holds
\begin{equation*} \tag{2.1.4}
\D^\lambda (fg) = \sum_{ \mu \in \Z_+^d(\lambda)}
C_\lambda^\mu \D^{\lambda -\mu} f \D^\mu g \in
\L_1^{\loc}(D).
\end{equation*}
\end{lemma}

\begin{remark} \label{z5}
For $ l \in \Z_+^d, m \in \N^d, \kappa \in \Z_+^d $ for an open set $ U \subset \R^d, $ a function $ f \in
\mathcal P_\kappa^{l,m,U} $
of the form (1.3.10), $ \lambda \in \Z_+^d(m) $ in the space of generalized functions on $ \R^d $ the relation holds
\begin{equation*} \tag{2.1.5}
\D^\lambda f = \sum_{\nu \in N_\kappa^{m,U}}
\D^\lambda (f_\nu g_{\kappa,\nu}^{m}) \in
L_1^{\loc}(\R^d),
\end{equation*}
where the sum, generally speaking, of a series in (2.1.5) is understood pointwise.
\end{remark}

\begin{propos}\label{p2.1.2}

Let $ l \in \Z_+^d, m \in \N^d, $ and let a domain $ D \subset \R^d $ and its open subset $ U \subset D $ be such that there exist constants $ \kappa^0 = \kappa^0(m,D,U) \in \Z_+^d, \gamma^0 = \gamma^0(m,D,U) \in \R_+^d, $
for which for any $ \kappa \in \Z_+^d $
there exists a mapping $ \nu_{\kappa} =
\nu_{\kappa}^{m,D,U}: N_{\kappa^0 +\kappa}^{m,U} \mapsto \Z^d, $
possessing the property that for each $ \nu \in N_{\kappa^0 +\kappa}^{m,U} $ the inclusion holds
\begin{equation*} \tag{2.1.6}
Q_{\kappa^0 +\kappa,\nu_{\kappa}(\nu)} \subset D \cap
(2^{-\kappa^0 -\kappa} \nu +\gamma^0 2^{-\kappa^0 -\kappa} B^d).
\end{equation*}
And let $ \lambda \in \Z_+^d(m), 1 \le p \le q \le \infty, $ and also, if the set $ U $ is bounded, let $ 1 \le q < p \le \infty. $ Then there exists a constant $ c_1(l,m,D,U,\lambda,p,q) > 0 $ such that for any function $ f \in L_p(D) $ and for $ \kappa \in \Z_+^d $ the inequality holds
\begin{multline*} \tag{2.1.7}
\| \D^\lambda E_{\kappa^0 +\kappa}^{l,m,D,U,
\nu_{\kappa}} f \|_{L_q(\R^d)} \le
c_1 2^{(\kappa, \lambda +(p^{-1} -q^{-1})_+ \e)}
\| f\|_{L_p(D)}, \\
(\text{see (2.1.3) with } \kappa^0 +\kappa \text{ instead of } \kappa, \nu_\kappa
\text{ instead of } \bm \nu_{\kappa^0 +\kappa}).
\end{multline*}
\end{propos}

The proof of Proposition \ref{p2.1.2} is given in [10], based on (2.1.3) -- (2.1.6), (1.3.6) and other facts.

The following proposition will be needed later.

\begin{propos}\label{p2.1.3}

Let $ l \in \Z_+^d, m \in \N^d, 1 \le p < \infty, $
and let a domain $ D \subset \R^d $ and its open subset $ U \subset D $ satisfy the conditions of Proposition \ref{p2.1.2}.
Then for any function $ f \in L_p(D) $ in $ L_p(U) $ the equality holds
\begin{equation*} \tag{2.1.8}
f \mid_U = \lim_{\mn(\kappa) \to \infty}
(E_{\kappa^0 +\kappa}^{l,m,D,U,
\nu_{\kappa}} f) \mid_U \text{ (see (2.1.3) and Proposition } \ref{p2.1.2}).
\end{equation*}
\end{propos}
The proof of Proposition \ref{p2.1.3} is analogous to the proof of Proposition 1.3.1 from [5] and is based on (1.3.7), (1.1.2) (see Proposition 2.1.3 from [10]).

To formulate the most important statement of this subsection, we introduce the following definitions.

Let $ D $ be a domain in $ \R^d $ and $ \kappa^0 \in \Z_+^d $ be such that (2.1.1) holds with $ \kappa^0 $ instead of $ \kappa, $ and let $ U \subset D $ be an open subset of $ D. $ Keeping in mind Remark \ref{z4} after (2.1.3), for $ m \in \N^d $ consider some family of mappings $ \Nu = \{ \nu_{\kappa}, \kappa \in \Z_+^d \}, $ of the form (2.1.2) with $ \kappa^0 +\kappa $ instead of $ \kappa, $ and $ \bm \nu_{\kappa^0 +\kappa} = \nu_\kappa, $ and for $ \kappa, l \in \Z_+^d, $ based on (2.1.3) and (1.3.16), define the linear operator $ \mathcal E_{\kappa^0, \kappa}^{l,m,D,U,\Nu}: \tilde L_1^{\loc}(D) \mapsto \mathcal P_{\kappa^0 +\kappa}^{l,m,U}, $ by setting
\begin{equation*} \tag{2.1.9}
\mathcal E_{\kappa^0, \kappa}^{l,m,D,U,\Nu} =
\sum_{\epsilon \in \Upsilon^d: \s(\epsilon) \subset
\s(\kappa)} (-\e)^\epsilon
H_{\kappa^0 +\kappa, \kappa^0 +\kappa -\epsilon}^{l,m,U}
E_{\kappa^0 +\kappa -\epsilon}^{l,m,D,U,\nu_{\kappa -\epsilon}}.
\end{equation*}
Here, taking into account (2.1.9), (2.1.3), (1.3.18), (1.3.19), (1.3.20), we have
\begin{multline*} \tag{2.1.10}
\mathcal E_{\kappa^0, \kappa}^{l,m,D,U,\Nu} f = \\
\sum_{\epsilon \in \Upsilon^d:
\s(\epsilon) \subset \s(\kappa)} (-\e)^\epsilon
\sum_{\nu \in N_{\kappa^0 +\kappa}^{m,U}}
\biggl(\sum_{\m^{\epsilon} \in \M_{\epsilon}^m(\nu)}
A_{\m^{\epsilon}}^m
S_{\kappa^0 +\kappa -\epsilon, \nu_{\kappa -\epsilon}
(\n_{\epsilon}(\nu,\m^{\epsilon}))}^{l} f\biggr)
g_{\kappa^0 +\kappa,\nu}^{m} = \\
\sum_{ \nu \in N_{\kappa^0 +\kappa}^{m,U}}
(U_{\kappa^0, \kappa,\nu}^{l,m,D,U,\Nu} f)
g_{\kappa^0 +\kappa, \nu}^{m},
f \in \tilde L_1^{\loc}(D),
\end{multline*}
where $ U_{\kappa^0, \kappa,\nu}^{l,m,D,U,\Nu}:
\tilde L_1^{\loc}(D) \mapsto
\mathcal P^{l} $ is a linear operator defined by
\begin{multline*} \tag{2.1.11}
U_{\kappa^0, \kappa,\nu}^{l,m,D,U,\Nu} f =
\sum_{\epsilon \in \Upsilon^d: \s(\epsilon) \subset
\s(\kappa)} (-\e)^\epsilon
\sum_{\m^{\epsilon} \in \M_{\epsilon}^m(\nu)}
A_{\m^{\epsilon}}^m
S_{\kappa^0 +\kappa -\epsilon, \nu_{\kappa -\epsilon}
(\n_{\epsilon}(\nu,\m^{\epsilon}))}^{l} f,\\ f \in \tilde L_1^{\loc}(D), \text{ see (1.3.12), (1.3.13), (1.3.15), (1.3.21))}.
\end{multline*}

\begin{definition} \label{d1}

For $ m \in \N^d $, we say that a domain $ D \subset \R^d $ and its open subset $ U \subset D $ form an $m$-regular pair if there exist constants $ \Kappa^0 = \Kappa^0(m,D,U) \in \Z_+^d, \Gamma^0 = \Gamma^0(m,D,U) \in \R_+^d, $ for which there exist families of mappings
\begin{multline*}
\Nu = \Nu^{m,D,U} = \{ \nu_{\kappa} = \
\nu_\kappa^{m,D,U}: N_{\Kappa^0 +\kappa}^{m,U} \mapsto \Z^d, \kappa \in \Z_+^d\}, \\
\{n_{\kappa}: N_{\Kappa^0 +\kappa}^{m,U} \mapsto \Z^d,
\kappa \in \Z_+^d\},
\end{multline*}
possessing the following properties:

1) for $ \kappa \in \Z_+^d $ for each $ \nu \in
N_{\Kappa^0 +\kappa}^{m,U} $
the inclusion holds
\begin{equation*} \tag{2.1.12}
(Q_{\Kappa^0 +\kappa,\nu_{\kappa}(\nu)} \cup
Q_{\Kappa^0 +\kappa,n_{\kappa}(\nu)}) \subset D \cap
(2^{-\Kappa^0 -\kappa} \nu +\Gamma^0 2^{-\Kappa^0 -\kappa} B^d);
\end{equation*}

2) for $ \kappa \in \Z_+^d, \nu \in
N_{\Kappa^0 +\kappa}^{m,U}, \epsilon \in \Upsilon^d:
\s(\epsilon) \subset \s(\kappa), \m^{\epsilon} \in
\M_{\epsilon}^m(\nu) $ for
$$
\mathcal D_{\kappa,\nu,\epsilon,\m^{\epsilon}} =
\bm x_{\kappa,\nu,\epsilon,\m^{\epsilon}} +
\bm \delta_{\kappa,\nu,\epsilon,\m^{\epsilon}} I^d,
$$
where the point $ \bm x_{\kappa,\nu,\epsilon,
\m^{\epsilon}} \in \R^d $ and the vector $ \bm \delta_{\kappa,\nu,\epsilon,\m^{\epsilon}} \in \R_+^d $
are defined by the equalities
\begin{multline*}
(\bm x_{\kappa,\nu,\epsilon,\m^{\epsilon}})_j =\\
\min(2^{-\Kappa^0_j -\kappa_j} (n_{\kappa}(\nu))_j,
2^{-\Kappa^0_j -\kappa_j +\epsilon_j} (\nu_{\kappa -\epsilon}(\n_{\epsilon}(\nu,\m^{\epsilon})))_j),
j \in \Nu_{1,d}; \text{ ( see (1.3.1))} \\
(\bm \delta_{\kappa,\nu,\epsilon,\m^{\epsilon}})_j =
\max(2^{-\Kappa^0_j -\kappa_j} (n_{\kappa}(\nu))_j
+2^{-\Kappa^0_j -\kappa_j}, \\
2^{-\Kappa^0_j -\kappa_j +\epsilon_j} (\nu_{\kappa -\epsilon}(\n_{\epsilon}(\nu,\m^{\epsilon})))_j +
2^{-\Kappa^0_j -\kappa_j +\epsilon_j})
-(\bm x_{\kappa,\nu,\epsilon,\m^{\epsilon}})_j,
j \in \Nu_{1,d},
\end{multline*}
the inclusion holds
\begin{equation*} \tag{2.1.13}
\mathcal D_{\kappa,\nu,\epsilon,\m^{\epsilon}} \subset D;
\end{equation*}

3) for $ \kappa \in \Z_+^d, \nu \in
N_{\Kappa^0 +\kappa}^{m,U} $ for any $ \epsilon \in \Upsilon^d: \s(\epsilon) \subset
\s(\kappa), $ and $ \m^{\epsilon}
\in \M_{\epsilon}^m(\nu) $ and for $ j \in \Nu_{1,d} \setminus \s(\epsilon) $ the equality holds
\begin{equation*} \tag{2.1.14}
(\nu_{\kappa -\epsilon}
(\n_{\epsilon}(\nu,\m^{\epsilon})))_j =
(\nu_{\kappa}(\nu))_j.
\end{equation*}
\end{definition}

\begin{remark} \label{z6}
Under the conditions of Definition \ref{d1}, there exists a constant $ \Gamma^1(m,D,U) \in \R_+^d $ such that for $ \kappa \in \Z_+^d, \nu \in N_{\Kappa^0 +\kappa}^{m,U},
\epsilon \in \Upsilon^d: \s(\epsilon) \subset
\s(\kappa), \m^{\epsilon} \in
\M_{\epsilon}^m(\nu) $ the following relations hold:

\begin{equation*} \tag{2.1.15}
2^{-\Kappa^0 -\kappa} \le \bm \delta_{\kappa,\nu,\epsilon,\m^{\epsilon}}
\le \Gamma^1 2^{-\Kappa^0 -\kappa},
\end{equation*}
\begin{equation*} \tag{2.1.16}
Q_{\Kappa^0 +\kappa, n_{\kappa}(\nu)} \cup
Q_{\Kappa^0 +\kappa -\epsilon, \nu_{\kappa -\epsilon}
(\n_{\epsilon}(\nu,\m^{\epsilon}))} \subset
\mathcal D_{\kappa,\nu,\epsilon,\m^{\epsilon}}.
\end{equation*}
\end{remark}
The verification of (2.1.15), (2.1.16) is carried out using (1.3.15), (1.3.13) and the conditions of Definition \ref{d1} (see [10]).

Note also that for $ m, \bm m \in \N^d: m \le \bm m, $
every $ \bm m $-regular pair $ (D,U)$ is an $m$-regular pair.
And if domains $ D \subset \mathcal D \subset \R^d $ and an open subset $ U \subset D $ are such that for some $ m \in
\N^d $, $ (D,U) $ is an $m$-regular pair, then $ (\mathcal D,U) $ is also an $m$-regular pair.

Let us give several examples of $m$-regular pairs.

Example 1.
The pair $ D = I^d, U = I^d $ is $m$-regular for $ m \in \N^d. $
In this example, for $ m \in \N^d $ the constant $ \Kappa^0(m,D,U) = 0, $ the set $ N_{\kappa}^{m,U} = \Nu_{-m, 2^\kappa -\e}, $ and the mappings are given by $ \nu_{\kappa}(\nu) =
n_{\kappa}(\nu) = \nu_+,
\nu \in N_{\kappa}^{m,U}, \kappa \in \Z_+^d. $

Example 2.

\begin{lemma}\label{l2.1.4}
Let a domain $ D \subset \R^d $ and its open subset $ U \subset D $ be such that there exists $ \delta \in \R_+^d $ for which the inclusion holds $ (U +\delta I^d) \subset D. $ Then $ (D,U) $ is an $m$-regular pair for any $ m \in \N^d. $
\end{lemma}

The proof of Lemma \ref{l2.1.4} is given in [10].

The following proposition occupies one of the central places in the paper.

\begin{propos}\label{p2.1.5}

Let $ m \in \N^d, $ and let a domain $ D \subset \R^d $ and its open subset $ U \subset D $ form an $m$-regular pair. And let $ l \in \N^d, \lambda \in \Z_+^d(m), 1 \le p < \infty, p \le q \le \infty, $ and also, if the set $ U $ is bounded, let $ 1 \le q < p. $
Then there exist constants $ c_{2}(l,m,D,U,\lambda,p,q) > 0, c_{3}(m,D,U) > 0 $ such that for $ \kappa \in \Z_+^d \setminus \{0\} $ and for $ f \in L_p(D) $ the inequality holds
\begin{multline*} \tag{2.1.17}
\| \D^\lambda \mathcal E_{\Kappa^0,\kappa}^{l -\e,m,D,U,\Nu} f \|_{L_q(\R^d)}
\le c_{2} 2^{(\kappa, \lambda +(p^{-1} -q^{-1})_+ \e)}
\Omega^{\prime l \chi_{\s(\kappa)}}(f,
(c_{3} 2^{-\kappa})^{\s(\kappa)})_{L_p(D)} \\
\text{(see (2.1.9) with $ \Kappa^0, \Nu $ from Definition \ref{d1}).}
\end{multline*}
\end{propos}

The proof of Proposition \ref{p2.1.5} is given in [10] using (2.1.4), (2.1.5), (2.1.10) -- (2.1.16), (1.3.4), (1.3.5), (1.3.8), (1.3.14), (1.3.21), (1.3.22).

\begin{propos}\label{p2.1.6}
Let the conditions of Proposition \ref{p2.1.5} hold.
Then if for a function $ f \in L_p(D) $ and for any nonempty set $ J \subset
\Nu_{1,d} $ the function
\begin{equation*} \tag{2.1.18}
\biggl(\prod_{j \in J}
t_j^{-\lambda_j -(p^{-1} -q^{-1})_+ -1}\biggr)
\Omega^{\prime l \chi_J}(f, c_{3} t^J)_{L_p(D)} \in
L_1((I^d)^J), \
\end{equation*}
then in $ L_q(U) $ the equality holds
\begin{equation*} \tag{2.1.19}
\D^\lambda (f \mid_U) = \sum_{\kappa \in \Z_+^d}
(\D^\lambda (\mathcal E_{\Kappa^0, \kappa}^{l -\e,m,D,U,\Nu} f)) \mid_U.
\end{equation*}
\end{propos}

The proof of Proposition \ref{p2.1.6} is given in [10]; it is based on (2.1.8), Lemma \ref{l1.2.1} and (1.3.17), (2.1.9).

\begin{propos}\label{p2.1.7}
Let $ m \in \N^d, $ a domain $ D \subset \R^d $ and its open subset $ U \subset D $ form an $m$-regular pair. Let also $ \alpha \in \R_+^d, 1 \le p < \infty, p \le q \le \infty $ and if the set $ U $ is bounded, let also $ 1 \le q < p, \lambda \in \Z_+^d(m) $ be such that the condition holds
\begin{equation*} \tag{2.1.20}
\alpha -\lambda -(p^{-1} -q^{-1})_+ \e >0.
\end{equation*}
Then

1) for any function $ f \in (S_p^\alpha H)^\prime(D) $ with $ l = l(\alpha) $
in $ L_q(U) $ the equality (2.1.19) holds;

2) there exists a constant $ c_{4}(\alpha,p,q,\lambda,m,D,U) > 0 $ such
that for $ f \in (S_p^\alpha H)^\prime(D) $ the inequality holds
\begin{equation*} \tag{2.1.21}
\| \D^\lambda (f \mid_U) \|_{L_q(U)} \le c_{4}
\| f \|_{(S_p^\alpha H)^\prime(D)}.
\end{equation*}
\end{propos}

The validity of item 1) follows from Proposition \ref{p2.1.6}, since in view of (2.1.20), (2.1.18) holds, and to obtain (2.1.21) it remains to use (2.1.19), and apply (2.1.7) (see (2.1.9), (1.3.16)) and (2.1.17) (see [10]).

Next, defining
$$
\Sigma^d = \{ \sigma \in \Z^d: \sigma_j \in \{-1,1\}, j =1,\ldots,d\},
$$
for $ \sigma \in \Sigma^d $ denote by $ \bm h_\sigma $ the mapping that assigns to each function $ f $ defined on some set $ S \subset \R^d $ the function $ \bm h_\sigma f $ defined on the set $ \{ x \in \R^d: \sigma x \in S\} = \sigma^{-1} S = \sigma S $ by the equality $ (\bm h_\sigma f)(x) = f(\sigma x). $

Let us note some useful properties of the mappings $ \bm h_\sigma $ (see [10]).
Since for $ \sigma \in \Sigma^d $ the mapping $ \R^d \ni x \mapsto \sigma x \in \R^d $ is bijective, the mapping $ \bm h_\sigma $ is a bijection of the set of all functions with domain in $ \R^d $ onto itself.
In this case, the inverse mapping $ \bm h_\sigma^{-1} $ for $ f: S \mapsto \R $ is given by
\begin{equation*} \tag{2.1.22}
(\bm h_\sigma^{-1} f)(x) = f(\sigma^{-1} x) =
f(\sigma x) = (\bm h_\sigma f)(x), x \in \sigma S.
\end{equation*}

For $ \sigma \in \Sigma^d $ for any sets $ S \subset S^\prime \subset \R^d $ and any function $ f: S^\prime \mapsto \R $ the equality holds
\begin{equation*} \tag{2.1.23}
\bm h_\sigma (f \mid_S) = (\bm h_\sigma f) \mid_{\sigma^{-1} S}.
\end{equation*}

For $ \sigma \in \Sigma^d $ for an open set $ D \subset \R^d, 1 \le p \le \infty $ and $ f \in L_p(D) $ the equality holds
\begin{equation*} \tag{2.1.24}
\| \bm h_\sigma f \|_{L_p(\sigma^{-1} D)} =
\| f \|_{L_p(D)},
\text{ and, hence, }
\bm h_\sigma \in \mathcal B(L_p(D),
L_p(\sigma^{-1} D)).
\end{equation*}

\begin{remark} \label{z7}
Note also that for $ \sigma \in \Sigma^d $, due to (2.1.24) and the equality
\begin{equation*}
\sigma Q_{\kappa,\nu} = Q_{\kappa,\sigma(\nu +\chi_J)}, \\
\text{ where } J = \{ j \in \Nu_{1,d}: \sigma_j = -1\}, \kappa \in \Z_+^d,
\nu \in \Z^d,
\end{equation*}
the equality holds
\begin{equation*}
\bm h_\sigma(\tilde L_1^{\loc}(D)) =
\tilde L_1^{\loc}(\sigma^{-1} D), \\
D  \text{ -- arbitrary open set in } \R^d.
\end{equation*}
\end{remark}

\begin{lemma}\label{l2.1.8}
Let $ \sigma \in \Sigma^d, D $ be an open set in $ \R^d, 1 \le p, q \le \infty, \lambda \in \Z_+^d $ and $ f \in L_p(D), \D^\lambda f \in L_q(D). $ Then the equality holds
\begin{equation*} \tag{2.1.25}
\D^\lambda (\bm h_\sigma f) = \sigma^\lambda \bm h_\sigma (\D^\lambda f).
\end{equation*}
\end{lemma}

\begin{lemma} \label{l2.1.9}
Let $ D $ be an open set in $ \R^d, 1 \le p < \infty,
\sigma \in \Sigma^d. $ Then for $ f \in L_p(D), l \in \N^d, J \subset \Nu_{1,d}: J \ne \emptyset, $ for $ t^J
\in (\R_+^d)^J $ the equality holds
\begin{equation*} \tag{2.1.26}
\Omega^{\prime l \chi_{J}} (\bm h_\sigma f,
t^{J})_{L_p(\sigma^{-1} D)} =
\Omega^{\prime l \chi_{J}} (f, t^{J})_{L_p(D)},
\end{equation*}
and
\begin{multline*} \tag{2.1.27}
\bm h_\sigma ((S_p^\alpha H)^\prime(D)) \subset (S_p^\alpha H)^\prime(\sigma^{-1} D),
\| \bm h_\sigma f \|_{(S_p^\alpha H)^\prime
(\sigma^{-1} D)} =
\| f \|_{(S_p^\alpha H)^\prime(D)},\\ f \in (S_p^\alpha H)^\prime(D),
\alpha \in \R_+^d.
\end{multline*}
\end{lemma}

Lemmas \ref{l2.1.8}, \ref{l2.1.9} are taken from [10].

Now the following proposition can be established.

\begin{propos}\label{p2.1.10}
For $ m \in \N^d $, let a domain $ D \subset \R^d $ and its open subset $ U \subset D $ be such that there exists $ \sigma \in
\Sigma^d $ for which the domain $ \sigma^{-1} D $ and its open subset $ \sigma^{-1} U \subset \sigma^{-1} D $ form an $m$-regular pair.
Then for any $ l \in \Z_+^d $ there exists a family of linear operators
$$
\mathfrak E_\kappa^{l,m,D,U}: \bm h_\sigma^{-1}(\tilde L_1^{\loc}(\sigma^{-1} D)) =
\tilde L_1^{\loc}(D) \mapsto L_1^{\loc}(\R^d), \kappa \in \Z_+^d,
$$
for which for $ 1 \le p < \infty, p \le q \le
\infty, $ and also, if the set $ U $ is bounded, also for $ 1 \le q < p $ the following holds:

1) for $ l \in \Z_+^d, \lambda \in \Z_+^d(m) $ there exists a constant $ c_5(l,m,D,U,\lambda,p,q) > 0 $ such that for $ f \in L_p(D) $ the inequality holds
\begin{equation*} \tag{2.1.28}
\| \D^\lambda \mathfrak E_0^{l,m,D,U} f \|_{L_q(\R^d)} \le c_5 \| f \|_{L_p(D)};
\end{equation*}

2) for $ l \in \N^d, \lambda \in \Z_+^d(m) $ there exist constants $ c_6(l,m,D,U,\lambda,p,q) > 0,
c_7(m,D,U) > 0 $ such that for $ \kappa \in \Z_+^d \setminus \{0\} $ and for $ f \in L_p(D) $ the inequality holds
\begin{multline*} \tag{2.1.29}
\| \D^\lambda \mathfrak E_{\kappa}^{l -\e,m,D,U} f
\|_{L_q(\R^d)} \le c_6 2^{(\kappa, \lambda +(p^{-1} -
q^{-1})_+ \e)} \Omega^{\prime l \chi_{\s(\kappa)}}(f, (c_7 2^{-\kappa})^{\s(\kappa)})_{L_p(D)};
\end{multline*}

3) for $ \alpha \in \R_+^d, \lambda \in \Z_+^d(m), $ satisfying condition (2.1.20),
for any function $ f \in (S_p^\alpha H)^\prime(D) $ and $ l = l(\alpha) $
in $ L_q(U) $ the equality holds
\begin{equation*} \tag{2.1.30}
\D^\lambda (f \mid_U) = \sum_{\kappa \in \Z_+^d}
(\D^\lambda (\mathfrak E_{\kappa}^{l -\e,m,D,U} f)) \mid_U,
\end{equation*}
and there exists a constant $ c_8(\alpha,m,D,U,\lambda,p,q) > 0 $ such that for any function $ f \in (S_p^\alpha H)^\prime(D) $ the inequality holds
\begin{equation*} \tag{2.1.31}
\| \D^\lambda (f \mid_U) \|_{L_q(U)} \le c_8
\| f \|_{(S_p^\alpha H)^\prime(D)}.
\end{equation*}
\end{propos}

\begin{proof}
Under the conditions of the proposition, taking into account Remark \ref{z7} before Lemma \ref{l2.1.8},
define for $ l \in \Z_+^d $ the family of operators
$$
\mathfrak E_\kappa^{l,m,D,U}:
\bm h_\sigma^{-1}(\tilde L_1^{\loc}(\sigma^{-1} D)) =
\tilde L_1^{\loc}(D) \mapsto L_1^{\loc}(\R^d), \kappa \in \Z_+^d,
$$
by setting (see (2.1.22), (2.1.3), (2.1.9))
\begin{multline*} \tag{2.1.32}
\mathfrak E_0^{l,m,D,U} = \bm h_\sigma^{-1}
\mathcal E_{\Kappa^0,0}^{l,m,\sigma^{-1} D,\sigma^{-1} U,\Nu} \bm h_\sigma =
\bm h_\sigma^{-1}
E_{\Kappa^0}^{l,m,\sigma^{-1} D,\sigma^{-1} U,\nu_{0}} \bm h_\sigma =\\
\bm h_\sigma E_{\Kappa^0}^{l,m,\sigma D,\sigma U,\nu_{0}} \bm h_\sigma,
\mathfrak E_\kappa^{l,m,D,U} = \bm h_\sigma^{-1}
\mathcal E_{\Kappa^0, \kappa}^{l,m,\sigma^{-1} D,
\sigma^{-1} U,\Nu} \bm h_\sigma =\\
\bm h_\sigma \mathcal E_{\Kappa^0, \kappa}^{l,m,\sigma D,\sigma U,\Nu}
\bm h_\sigma, \kappa \in \Z_+^d \setminus \{0\},
\end{multline*}
where
\begin{multline*}
\Kappa^0 = \Kappa^0(m,\sigma^{-1} D, \sigma^{-1} U) \in \Z_+^d,
\Nu = \Nu^{m,\sigma^{-1} D,\sigma^{-1} U} =\\
\{ \nu_\kappa = \nu_{\kappa}^{m,\sigma^{-1} D,
\sigma^{-1} U}:
N_{\Kappa^0 +\kappa}^{m,\sigma^{-1} U} \mapsto \Z^d, \kappa \in \Z_+^d\}
\end{multline*}
are the objects from Definition \ref{d1}.

Let us verify the validity of relations (2.1.28) -- (2.1.31).
Under the conditions of item 1), using (2.1.32), (2.1.25), (2.1.24), (2.1.7) and again (2.1.24), for $ f \in L_p(D) $ we derive (2.1.28).

Next, under the conditions of item 2), applying (2.1.32), (2.1.25), (2.1.24), (2.1.17), (2.1.26), for $ \kappa \in \Z_+^d \setminus \{0\} $ and for $ f \in L_p(D) $ we obtain (2.1.29).

Finally, under the conditions of item 3), in view of (2.1.27), according to item 1) of Proposition \ref{p2.1.7}
for $ f \in (S_p^\alpha H)^\prime(D) $ and $ l = l(\alpha) $ in $ L_q(\sigma^{-1} U) $ the equality holds
\begin{equation*}
\D^\lambda ((\bm h_\sigma f) \mid_{\sigma^{-1} U}) = \sum_{\kappa \in \Z_+^d}
(\D^\lambda (\mathcal E_{\Kappa^0, \kappa}^{l -\e,m,\sigma^{-1} D,\sigma^{-1} U,\Nu}
(\bm h_\sigma f))) \mid_{\sigma^{-1} U},
\end{equation*}
from which, due to (2.1.24), we conclude that in $ L_q(\sigma^{-1} \sigma^{-1} U) =
L_q(U) $ the equality holds
\begin{multline*} \tag{2.1.33}
\bm h_\sigma (\D^\lambda ((\bm h_\sigma f) \mid_{\sigma^{-1} U})) = \\
\sum_{\kappa \in \Z_+^d} \bm h_\sigma
((\D^\lambda (\mathcal E_{\Kappa^0, \kappa}^{l -\e,m,\sigma^{-1} D,\sigma^{-1} U,\Nu}
(\bm h_\sigma f))) \mid_{\sigma^{-1} U}).
\end{multline*}

Using (2.1.25), (2.1.23), (2.1.22), we have
\begin{equation*} \tag{2.1.34}
\bm h_\sigma (\D^\lambda ((\bm h_\sigma f) \mid_{\sigma^{-1} U})) =
\sigma^{\lambda} \D^\lambda (f \mid_{U}).
\end{equation*}
Applying (2.1.23), (2.1.25), (2.1.32), we obtain
\begin{equation*} \tag{2.1.35}
\bm h_\sigma
((\D^\lambda (\mathcal E_{\Kappa^0, \kappa}^{l -\e,m,\sigma^{-1} D,\sigma^{-1} U,\Nu}
(\bm h_\sigma f))) \mid_{\sigma^{-1} U}) =
\sigma^{\lambda} (\D^\lambda (\mathfrak E_{\kappa}^{l -\e,m,D,U} f)) \mid_U.
\end{equation*}
Substituting (2.1.34) and (2.1.35) into (2.1.33), in $ L_q(U) $ we arrive at the equality
\begin{equation*}
\sigma^\lambda \D^\lambda (f \mid_U) = \sum_{\kappa \in \Z_+^d}
\sigma^\lambda (\D^\lambda (\mathfrak E_{\kappa}^{l -\e,m,D,U} f)) \mid_U,
\end{equation*}
from which (2.1.30) follows.

It remains to verify the validity of (2.1.31). To do this, under the conditions of item 3) of the proposition, taking into account (2.1.27), based on (2.1.21), we conclude that for $ f \in (S_p^\alpha H)^\prime(D) $ the relation holds
\begin{equation*}
\| \D^\lambda ((\bm h_\sigma f) \mid_{\sigma^{-1} U}) \|_{L_q(\sigma^{-1} U)}
\le c_4 \| \bm h_\sigma f \|_{(S_p^\alpha H)^\prime(\sigma^{-1} D)}.
\end{equation*}
From here, taking into account (2.1.27), and also the fact that due to (2.1.24), (2.1.34) the relation holds
$$
\| \D^\lambda ((\bm h_\sigma f) \mid_{\sigma^{-1} U}) \|_{L_q(\sigma^{-1} U)} = \\
\| \D^\lambda (f \mid_{U}) \|_{L_q(U)},
$$
we obtain (2.1.31).
\end{proof}

\begin{remark}\label{z8}
If for a domain $ D \subset \R^d $ and its open subset $ U \subset D $
there exist $ \delta \in \R_+^d $ and $ \sigma \in \Sigma^d, $ for which $ (U +\sigma \delta I^d) \subset D, $ then $ (\sigma^{-1} U +\delta I^d) \subset
\sigma^{-1} D, $ and, consequently, by Lemma \ref{l2.1.4}, $ (\sigma^{-1} D, \sigma^{-1} U) $ is an $m$-regular pair for any $ m \in \N^d. $
\end{remark}

Relying on (2.1.31), it is easy to obtain a proof of Theorem \ref{t2.1.11} (see [10]).

\begin{theorem}\label{t2.1.11}
For $ m \in \N^d $, let $ D $ be a domain in $ \R^d, $ for which there exists a system of open subsets $ \{ U_i \subset D, i =1,\ldots,\mathcal I\} $

such that for $ i =1,\ldots,\mathcal I $ there exists $ \sigma^i \in \Sigma^d, $
for which the domain $ (\sigma^i)^{-1} D $ and its open subset $ (\sigma^i)^{-1} U_i \subset (\sigma^i)^{-1} D $ form an $m$-regular pair and $ D = \cup_{i =1}^{\mathcal I} U_i. $
Then for $ \alpha \in \R_+^d, 1 \le p < \infty, p \le q \le \infty, $ and also, if $ D $ is a bounded domain, also for $ 1 \le q < p, \lambda \in
\Z_+^d(m), $ satisfying condition (2.1.20), there exists a constant $ c_{9}(\alpha,p,q,\lambda,m,D) > 0 $ such that for any function $ f \in (S_p^\alpha H)^\prime(D) $ the inequality holds
\begin{equation*}
\| \D^\lambda f \|_{L_q(D)} \le c_{9}
\| f \|_{(S_p^\alpha H)^\prime(D)}.
\end{equation*}
\end{theorem}

\subsection{} 
In this subsection, an upper estimate of the quantity indicated in the title of the section will be obtained.

Let us recall the formulation of the general problem, a special case of which is the problem considered in this paper.

Let $ X,Y $ be Banach spaces, $ \mathcal B(X,  Y)  $ be the Banach space of continuous linear operators $ T:  X \mapsto Y $ with the usual norm
$$
\|T\|_{\mathcal B(X, Y)} =\sup_{ x \in B(X)} \|T x\|_Y,
$$

$ U : D(U) \mapsto Y $ be a linear operator with domain $ D(U) \subset X. $ Let also a set $ K \subset D(U) $ and $ \rho >0. $

It is required to describe the behavior, depending on $ \rho $, of the quantity
$$
E(U,X,Y,K,\rho) =\inf_{\{ T  \in  \mathcal B(X,Y):
\|T\|_{\mathcal B(X,Y)} \le \rho\}} \sup_{x \in K}
\|U x -T x\|_Y.
$$

\begin{lemma}\label{l2.2.1}

For $ m \in \N^d $, let $ D $ be a domain in $ \R^d, $ for which there exists a system of open subsets $ \{ U_i \subset D, i =1,\ldots,\mathcal I\} $
such that for $ i =1,\ldots,\mathcal I $ there exists $ \sigma^i \in \Sigma^d, $
for which the domain $ (\sigma^i)^{-1} D $ and its open subset $ (\sigma^i)^{-1} U_i \subset (\sigma^i)^{-1} D $ form an $m$-regular pair and $ D = \cup_{i =1}^{\mathcal I} U_i. $
Let also for $ \alpha \in \R_+^d, \lambda \in
\Z_+^d(m) $ for $ 1 \le p < \infty,
p \le q \le \infty, $ and also, if $ D $ is a bounded domain, also for $ 1 \le q < p $ the condition (2.1.20) holds,
and besides, for $ 1 \le s < \infty, s \le q \le \infty, $ and also, if $ D $ is a bounded domain, also for $ 1 \le q < s $ the condition $ \lambda +(s^{-1} -q^{-1})_+ \e >0, 1 \le \theta \le \infty $ holds. Put $ \gamma = \alpha -\lambda -(p^{-1} -q^{-1})_+ \e $ and $ \tau = \lambda +(s^{-1} -q^{-1})_+ \e. $
Then there exist constants
$ c_1(\alpha,p,\theta,\lambda,q,s,m,D) > 0 $ and
$ c_2(\alpha,p,\theta,\lambda,q,s,m,D) >0 $ such that for any $ r \in \N $ one can construct a linear operator
$$
T_r = T_r^{\alpha,\lambda,m,D}: \tilde L_1^{\loc}(D) \mapsto L_1^{\loc}(\R^d),
$$
possessing the following properties:

1) for any function $ f \in (\mathcal S_{p,\theta}^\alpha
\mathcal B)^\prime(D) $ the inequality holds
\begin{equation*} \tag{2.2.1}
\| \D^\lambda f -T_r f \|_{L_q(D)} \le
c_1 2^{-\mn(\tau^{-1} \gamma) r}
r^{(\cmn(\tau^{-1} \gamma) -1) (1 -1/\theta)},
\end{equation*}

2) the inequality holds
\begin{equation*} \tag{2.2.2}
\| T_r \|_{\mathcal B(L_s(D), L_q(D))} \le
c_2 2^r r^{\cmn(\tau^{-1} \gamma) -1}.
\end{equation*}
\end{lemma}

\begin{proof}
Under the conditions of the lemma, denote by $ J $ the set $ J = \{ j \in \Nu_{1,d}:
\tau_j^{-1} \gamma_j = \mn(\tau^{-1} \gamma)\} $ and define the vector $ \beta \in \R_+^d, $ by setting $ \beta_j = \tau_j $ for $ j \in J, $ and for $ j \in \Nu_{1,d} \setminus J $ choosing $ \beta_j $ so as to satisfy the conditions $ \beta_j > \tau_j $ and $ \beta_j^{-1}
\gamma_j > \mn(\tau^{-1} \gamma). $

From the definition of the vector $ \beta $ it is clear that
$$
\mn(\beta^{-1} \gamma) = \mn(\tau^{-1} \gamma),
$$
and
$$
\cmn(\beta^{-1} \gamma) = \cmn(\tau^{-1} \gamma),
$$
$$
\mx(\beta^{-1} \tau) =1,
$$
$$
\cmx(\beta^{-1} \tau) = \cmn(\tau^{-1} \gamma).
$$

Next, setting $ l = l(\alpha) $ and taking into account that $ \lambda \in \Z_+^d(m), $ define for $ r \in \N, i =1,\ldots,\mathcal I $
the linear operator $ T_r^i: \tilde L_1^{\loc}(D) \mapsto L_1^{\loc}(\R^d), $ by setting
\begin{equation*}
T_r^i = \sum_{\kappa \in \Z_+^d: (\kappa, \beta) \le r}
\D^\lambda \mathfrak E_{\kappa}^{l -\e,m,D,U_i}.
\text{ (see Proposition } \ref{p2.1.10}.)
\end{equation*}

Then for $ r \in \N, i =1,\ldots,\mathcal I $ for $ f \in (\mathcal S_{p,\theta}^\alpha \mathcal B)^\prime(D), $ taking into account the validity of (2.1.20), on the basis of Proposition \ref{p2.1.10}, taking into account (1.1.4), using (2.1.30) and applying (2.1.29), we obtain
\begin{multline*} \tag{2.2.3}
\| (\D^\lambda f) \mid_{U_i} -(T_r^i f) \mid_{U_i}
\|_{L_q(U_i)} = \| \D^\lambda (f \mid_{U_i}) -
(\sum_{\kappa \in \Z_+^d: (\kappa, \beta) \le r}
\D^\lambda (\mathfrak E_{\kappa}^{l -\e,m,D,U_i} f))
\mid_{U_i}\|_{L_q(U_i)} = \\
\| \D^\lambda (f \mid_{U_i}) -
\sum_{\kappa \in \Z_+^d: (\kappa, \beta) \le r}
(\D^\lambda (\mathfrak E_{\kappa}^{l -\e,m,D,U_i} f))
\mid_{U_i}\|_{L_q(U_i)} = \\
\| \sum_{\kappa \in \Z_+^d: (\kappa, \beta) > r}
(\D^\lambda (\mathfrak E_{\kappa}^{l -\e,m,D,U_i} f)) \mid_{U_i} \|_{L_q(U_i)} \le \\
\sum_{\kappa \in \Z_+^d: (\kappa, \beta) > r}
\| (\D^\lambda (\mathfrak E_{\kappa}^{l -\e,m,D,U_i} f)) \mid_{U_i} \|_{L_q(U_i)} \le \\
\sum_{\kappa \in \Z_+^d: (\kappa, \beta) > r}
\| \D^\lambda (\mathfrak E_{\kappa}^{l -\e,m,D,U_i} f)
\|_{L_q(\R^d)} \le \\
\sum_{\kappa \in \Z_+^d: (\kappa, \beta) > r}
c_3 2^{(\kappa, \lambda +(p^{-1} -q^{-1})_+ \e)}
\Omega^{\prime l \chi_{\s(\kappa)}}(f,
(c_4 2^{-\kappa})^{\s(\kappa)})_{L_p(D)} = \\
c_3 \sum_{\kappa \in \Z_+^d: (\kappa, \beta) > r}
2^{(\kappa, \lambda +(p^{-1} -q^{-1})_+ \e)}
\Omega^{\prime l \chi_{\s(\kappa)}}(f,
(c_4 2^{-\kappa})^{\s(\kappa)})_{L_p(D)}.
\end{multline*}

Estimating the sum on the right-hand side of (2.2.3), using H\"older's inequality for $ r \in \N $
for $ f \in (\mathcal S_{p,\theta}^\alpha
\mathcal B)^\prime(D) $ we obtain
\begin{multline*} \tag{2.2.4}
\sum_{ \kappa \in \Z_+^d: (\kappa, \beta) > r}
2^{(\kappa, \lambda +(p^{-1} -q^{-1})_+ \e)}
\Omega^{\prime l \chi_{\s(\kappa)}}(f,
(c_{4} 2^{-\kappa})^{\s(\kappa)})_{L_p(D)} =\\
\sum_{ \kappa \in \Z_+^d: (\kappa, \beta) > r}
2^{-(\kappa, \alpha -\lambda -(p^{-1} -q^{-1})_+ \e)}
2^{(\kappa, \alpha)} \Omega^{\prime l \chi_{\s(\kappa)}}(f, (c_{4} 2^{-\kappa})^{\s(\kappa)})_{L_p(D)} \le\\
\biggl(\sum_{ \kappa \in \Z_+^d: (\kappa, \beta) > r}
2^{-(\kappa, \gamma) \theta^\prime}
\biggr)^{1/\theta^\prime} \times \\
\biggl(\sum_{ \kappa \in \Z_+^d: (\kappa, \beta) > r}
(2^{(\kappa, \alpha)}
\Omega^{\prime l \chi_{\s(\kappa)}}(f,
(c_{4} 2^{-\kappa})^{\s(\kappa)})_{L_p(D)})^\theta \biggr)^{1/\theta},
\theta^\prime = \theta /(\theta -1).
\end{multline*}

Thanks to the validity of (2.1.20), using (1.2.1), we derive
\begin{multline*} \tag{2.2.5}
\biggl(\sum_{ \kappa \in \Z_+^d: (\kappa, \beta) > r}
2^{-(\kappa, \gamma) \theta^\prime}
\biggr)^{1/\theta^\prime} \le\\
(c_{5} 2^{-\mn(\theta^\prime \beta^{-1} \gamma) r}
r^{\cmn(\theta^\prime
\beta^{-1} \gamma) -1})^{1/\theta^\prime} =\\
(c_{5} 2^{-\theta^\prime \mn(\beta^{-1} \gamma) r}
r^{\cmn(\beta^{-1} \gamma) -1})^{1/\theta^\prime} =\\
c_{6} 2^{-\mn(\beta^{-1} \gamma) r}
r^{(\cmn(\beta^{-1} \gamma) -1) /\theta^\prime} =\\
c_{6} 2^{-\mn(\tau^{-1} \gamma) r}
r^{(\cmn(\tau^{-1} \gamma) -1) (1 -1/\theta)}, r \in \N.
\end{multline*}

As indicated in [10, (see (2.3.3))], there exists a constant $ c_{7}(\alpha,p,\theta) >0 $ such that for $ f \in
(\mathcal S_{p,\theta}^\alpha \mathcal B)^\prime(D) $ for $ r \in \N $ the inequality holds
\begin{equation*} \tag{2.2.6}
\biggl( \sum_{\kappa \in \Z_+^d: (\kappa, \beta) > r}
(2^{(\kappa,\alpha)}
\Omega^{\prime l \chi_{\s(\kappa)}}(f,
(c_{4} 2^{-\kappa})^{\s(\kappa)})_{L_p(D)})^\theta \biggr)^{1/\theta} \le c_{7}.
\end{equation*}

Substituting (2.2.5) and (2.2.6) into (2.2.4), we arrive at the inequality
\begin{multline*} \tag{2.2.7}
\sum_{ \kappa \in \Z_+^d: (\kappa, \beta) > r}
2^{(\kappa, \lambda +(p^{-1} -q^{-1})_+ \e)}
\Omega^{\prime l \chi_{\s(\kappa)}}(f, (c_{4}
2^{-\kappa})^{\s(\kappa)})_{L_p(D)} \le\\
c_{8} 2^{-\mn(\tau^{-1} \gamma) r}
r^{(\cmn(\tau^{-1} \gamma) -1) (1 -1/\theta)},
f \in (\mathcal S_{p,\theta}^\alpha
\mathcal B)^\prime(D), r \in \N.
\end{multline*}

Combining (2.2.3) with (2.2.7), we arrive at the estimate
\begin{multline*} \tag{2.2.8}
\| (\D^\lambda f) \mid_{U_i} -(T_r^i f) \mid_{U_i}
\|_{L_q(U_i)} \le \\
c_{9} 2^{-\mn(\tau^{-1} \gamma) r}
r^{(\cmn(\tau^{-1} \gamma) -1) (1 -1/\theta)},
f \in (\mathcal S_{p,\theta}^\alpha
\mathcal B)^\prime(D), r \in \N, i =1,\ldots,\mathcal I.
\end{multline*}

At the same time, for $ f \in L_s(D) $ for $ r \in \N,
i =1,\ldots,\mathcal I $, due to (2.1.28), (2.1.29) and (1.2.2), the inequality holds
\begin{multline*} \tag{2.2.9}
\| T_r^i f\|_{L_q(D)} = \| \sum_{\kappa \in \Z_+^d: (\kappa, \beta) \le r}
\D^\lambda \mathfrak E_{\kappa}^{l -\e,m,D,U_i} f
\|_{L_q(D)} \le \\
\sum_{\kappa \in \Z_+^d: (\kappa, \beta) \le r}
\| \D^\lambda \mathfrak E_{\kappa}^{l -\e,m,D,U_i} f
\|_{L_q(D)} \le \\
\| \D^\lambda \mathfrak E_{0}^{l -\e,m,D,U_i} f
\|_{L_q(\R^d)} +
\sum_{\kappa \in \Z_+^d \setminus \{0\}:
(\kappa, \beta) \le r}
\| \D^\lambda \mathfrak E_{\kappa}^{l -\e,m,D,U_i} f
\|_{L_q(\R^d)} \le \\
c_{10} \| f \|_{L_s(D)} +\sum_{\kappa \in \Z_+^d \setminus \{0\}: (\kappa, \beta) \le r}
c_{11} 2^{(\kappa, \lambda +(s^{-1} -q^{-1})_+ \e)}
\Omega^{\prime l \chi_{\s(\kappa)}}(f, (c_{12}
2^{-\kappa})^{\s(\kappa)})_{L_s(D)} \le \\
\sum_{ \kappa \in \Z_+^d: (\kappa,\beta) \le r}
c_{13} 2^{(\kappa, \tau)} \|f\|_{L_s(D)} \le
c_{14} 2^{ \mx(\beta^{-1} \tau) r} r^{ \cmx(\beta^{-1} \tau) -1} \| f \|_{L_s(D)} =\\
c_{14} 2^r r^{\cmn(\tau^{-1} \gamma) -1}
\|f\|_{L_s(D)}.
\end{multline*}

Next, denoting by $ \bm u^i = U_i \setminus (\cup_{j =1}^{i -1}
U_j), \chi^i = \chi_{\bm u^i},
i =1,\ldots,\mathcal I, $ and
taking into account that $ \cup_{i =1}^{\mathcal I} \bm u^i = D, $
we see that for $ x \in D $ the equality holds $ \sum_{i =1}^{\mathcal I} \chi^i(x) =1. $
Now for $ r \in \N $ construct the operator $ T_r:
\tilde L_1^{\loc}(D) \mapsto L_1^{\loc}(\R^d), $
defining its value by
\begin{equation*}
T_r f = \sum_{i =1}^{\mathcal I} \chi^i
T_r^i f, f \in \tilde L_1^{\loc}(D).
\end{equation*}

Then for $ r \in \N, $ taking into account (2.2.8),
for $ f \in (\mathcal S_{p,\theta}^\alpha
\mathcal B)^\prime(D) $ we have
\begin{multline*}
\| \D^\lambda f -T_r f\|_{L_q(D)} =
\| (\sum_{i =1}^{\mathcal I} \chi^i) \D^\lambda f -
\sum_{i =1}^{\mathcal I} \chi^i T_r^i f\|_{L_q(D)} = \\
\| \sum_{i =1}^{\mathcal I} \chi^i (\D^\lambda f -
T_r^i f )\|_{L_q(D)} \le \\
\sum_{i =1}^{\mathcal I} \| \chi^i (\D^\lambda f -
T_r^i f)\|_{L_q(D)} = \\
\sum_{i =1}^{\mathcal I} \| (\chi^i (\D^\lambda f -
T_r^i f)) \mid_{U_i} \|_{L_q(U_i)} \le \\
\sum_{i =1}^{\mathcal I} \| (\D^\lambda f -
T_r^i f) \mid_{U_i} \|_{L_q(U_i)} = \\
\sum_{i =1}^{\mathcal I} \| (\D^\lambda f) \mid_{U_i} -
(T_r^i f) \mid_{U_i} \|_{L_q(U_i)} \le \\
\sum_{i =1}^{\mathcal I}
c_{9} 2^{-\mn(\tau^{-1} \gamma) r} r^{(\cmn(\tau^{-1} \gamma) -1) (1 -1/\theta)} \le
c_1 2^{-\mn(\tau^{-1} \gamma) r}
r^{(\cmn(\tau^{-1} \gamma) -1) (1 -1/\theta)},
\end{multline*}
which coincides with (2.2.1);
and for $ f \in L_s(D) $, due to (2.2.9), the inequality holds
\begin{multline*}
\| T_r f \|_{L_q(D)} =
\| \sum_{i =1}^{\mathcal I} \chi^i T_r^i f\|_{L_q(D)} \le \\
\sum_{i =1}^{\mathcal I} \| \chi^i T_r^i f\|_{L_q(D)} \le \\
\sum_{i =1}^{\mathcal I} \| T_r^i f\|_{L_q(D)}
\le
\sum_{i =1}^{\mathcal I} c_{14} 2^r r^{\cmn(\tau^{-1}\gamma) -1} \|f\|_{L_s(D)} \le
c_{2} 2^r r^{\cmn(\tau^{-1} \gamma) -1}
\|f\|_{L_s(D)},
\end{multline*}
which implies (2.2.2).
\end{proof}

\begin{theorem}\label{t2.2.2}

Let the conditions of Lemma \ref{l2.2.1} hold and for $ s > p $ the inequality holds
\begin{equation*} \tag{2.2.10}
\alpha -(p^{-1} -s^{-1})_+ \e >0.
\end{equation*}
Let also $ U = \D^\lambda, D(U)
= \{f \in L_s(D): \D^\lambda f \in L_q(D)\}, X = L_s(D), Y = L_q(D), K = (\mathcal S_{p,\theta}^\alpha
\mathcal B)^\prime(D). $
Then there exist constants $ c_{15}(U,X,Y,K) >0 $ and $ \rho_0(U,X,Y,K) >0 $ such that for $ \rho \ge
\rho_0 $ the inequality holds
\begin{equation*} \tag{2.2.11}
E(U,X,Y,K,\rho) \le c_{15} \rho^{-\mn(\tau^{-1} \gamma)}
(\log \rho)^{(\mn(\tau^{-1} \gamma) +1 -1 /\theta)
(\cmn(\tau^{-1} \gamma) -1)}.
\end{equation*}
\end{theorem}
\begin{proof}
Inequality (2.2.11), taking into account (2.1.20), (2.2.10) (see also Theorem \ref{t2.1.11} and (1.1.4)), is a simple consequence of (2.2.1) and (2.2.2).
\end{proof}

\section{\S 3. Lower estimate of the best accuracy of approximation in $ L_q(D) $ of the operator $ \D^\lambda $
by bounded operators acting from $ L_s(D) $ into $ L_q(D) $, on classes $ B((S_{p,\theta}^\alpha B)^\prime(D)) $}

\subsection{} In this subsection, a lower estimate of the quantity under study is established.

For $ \delta \in \R_+^d $ and $ x^0 \in \R^d $, denote by $ h_{\delta, x^0} $ the mapping that assigns to each function $ f $ defined on some set $ S \subset \R^d $ the function $ h_{\delta, x^0} f $ defined on the set $ \{ x \in \R^d: x^0 +\delta
x \in S\} = \delta^{-1} (S -x^0) $ by the equality $ (h_{\delta, x^0} f)(x) =
f(x^0 +\delta x). $
Since for $ \delta \in \R_+^d, x^0 \in \R^d $ the mapping $ \R^d \ni x \mapsto x^0 +\delta x \in \R^d $ is bijective, the mapping $ h_{\delta, x^0} $ is a bijection of the set of all functions with domain in $ \R^d $ onto itself.
In this case, the inverse mapping $ h_{\delta, x^0}^{-1} $ is given by
\begin{equation*} \tag{3.1.1}
(h_{\delta, x^0}^{-1} f)(x) = f(\delta^{-1} (x -x^0)) =
(h_{\delta^\prime, x^{\prime 0}} f)(x)  \text{ with } \delta^\prime = \delta^{-1},
x^{\prime 0} =-\delta^{-1} x^0.
\end{equation*}

Note that for $ 1 \le p \le \infty $ for $ f \in
L_p(x^0 +\delta D), $ where $ D $ is a domain in $ \R^d, \delta \in \R_+^d, x^0 \in \R^d, $ the equality holds
\begin{equation*} \tag{3.1.2}
\| h_{\delta,x^0} f\|_{L_p(D)} = \delta^{-p^{-1} \e}
\|f\|_{L_p(x^0 +\delta D)},
\end{equation*}
and, consequently, for $ f \in L_p(D) $ the equality holds
\begin{equation*} \tag{3.1.3}
\| h_{\delta,x^0}^{-1} f\|_{L_p(x^0 +\delta D)} = \delta^{p^{-1} \e} \|f\|_{L_p(D)}.
\end{equation*}

In [13], the following statement is established.

\begin{lemma} \label{l3.1.1}

Let $ l \in \N^d, D $ be a domain in $ \R^d, 1 \le p < \infty,
\delta \in \R_+^d, x^0 \in \R^d. $ Then for $ J \subset \{1,\ldots,d\}: J \ne \emptyset, t \in \R_+^d $ for $ f \in L_p(x^0 +\delta D) $ the equality holds
\begin{equation*} \tag{3.1.4}
\Omega^{\prime l \chi_J}((h_{\delta, x^0} f),
t^J)_{L_p(D)} = \delta^{-p^{-1} \e}
\Omega^{\prime l \chi_J}(f, (\delta t)^J)_{L_p(x^0 +\delta D)}.
\end{equation*}
\end{lemma}

\begin{lemma} \label{l3.1.2}

Let $ D $ be a domain in $ \R^d, \alpha \in \R_+^d,
1 \le p < \infty, 1 \le \theta \le \infty, \delta \in \R_+^d, x^0 \in \R^d. $ Then there exist constants $ c_1(\alpha,p,\delta) > 0,
c_2(\alpha,p,\delta) > 0 $
such that for any function $ f \in
(S_{p,\theta}^\alpha B)^\prime(x^0 +\delta D) $
the inequality holds
\begin{equation*} \tag{3.1.5}
\| h_{\delta, x^0} f \|_{(S_{p,\theta}^\alpha
B)^\prime(D)} \le c_1 \| f \|_{(S_{p,\theta}^\alpha
B)^\prime(x^0 +\delta D)},
\end{equation*}
and for $ f \in (S_{p,\theta}^\alpha B)^\prime(D) $ the inequality holds
\begin{equation*} \tag{3.1.6}
\| h_{\delta, x^0}^{-1} f \|_{(S_{p,\theta}^\alpha
B)^\prime(x^0 +\delta D)} \le
c_2 \| f \|_{(S_{p,\theta}^\alpha B)^\prime(D)}.
\end{equation*}
\end{lemma}

Inequality (3.1.5) follows from (3.1.2) and (3.1.4). And (3.1.6) follows from (3.1.1) and (3.1.5). For more details, see [13].

For the formulation of the following statement, we will use the following notation. For a set $ S $ consisting of functions $ f $ whose domain of definition contains a set $ D \subset \R^d, $ let $ S \mid_D $ denote the set $ S \mid_D =
\{ f \mid_D: f \in S\}. $

For $ m \in \Z_+^d $ and a domain $ D \subset \R^d $, denote by $ C^m(D) $ the space of all functions $ f: D \mapsto \R $ for which for each $ \lambda \in \Z_+^d(m) $ there exists a continuous partial derivative $ \D^\lambda f $ of order $ \lambda $ in the domain $ D, $ and by $ C_0^m(d) $ denote the space of all functions $ f \in C^m(\R^d) $ whose support $ \supp f \subset D. $

\begin{lemma} \label{l3.1.3}

Under the conditions of Theorem \ref{t2.2.2} for $ \bm m \in \N^d: \bm m \ge m, $ there exist constants $ c_3(\alpha,\lambda,p,q,s,D) > 0 $ and $ c_4(\lambda,s,q,D) > 0 $ such that for $ \rho > 0 $ the inequality holds
\begin{multline*} \tag{3.1.7}
E(\D^\lambda,L_s(D),L_q(D),B((S_{p,\theta}^\alpha
B)^\prime(D)),\rho) \ge \\
c_{3} E(\D^\lambda,L_s(I^d), L_Q(I^d),(B((S_{p,\theta}^\alpha
B)^\prime(\R^d)) \cap C_0^{\bm m}(I^d)) \mid_{I^d}, c_{4} \rho).
\end{multline*}
\end{lemma}

\begin{proof}
Fix $ x^0 \in \R^d $ and $ \delta \in \R_+^d $ such that $ Q =
(x^0 +\delta I^d) \subset D. $

Next, note that for $ f \in L_s(I^d): \D^\lambda f \in L_q(I^d),
T \in \mathcal B(L_s(Q),L_q(Q)) $, taking into account (3.1.2), the relation holds
\begin{multline*}
\| \D^\lambda f -(h_{\delta,x^0} T h_{\delta,x^0}^{-1}) f\|_{L_q(I^d)} = \\
\delta^{\lambda -q^{-1} \e} \| \D^\lambda
(h_{\delta,x^0}^{-1} f) -
\delta^{-\lambda} T (h_{\delta,x^0}^{-1} f)\|_{L_q(Q)}.
\end{multline*}
Note also that for $ T \in \mathcal B(L_s(Q),L_q(Q)), $ due to (3.1.2), (3.1.3), the norm
\begin{multline*}
\|h_{\delta,x^0} T h_{\delta,x^0}^{-1}
\|_{\mathcal B(L_s(I^d),L_q(I^d))} \le \\
\delta^{s^{-1} \e -q^{-1} \e}
\| T\|_{\mathcal B(L_s(Q),L_q(Q))}.
\end{multline*}
Taking these circumstances into account, as well as (3.1.1), (3.1.6), for $ \bm m \in \N^d, \rho \in \R_+ $
for $ T \in \mathcal B(L_s(Q),L_q(Q)):
\|T\|_{\mathcal B(L_s(Q),L_q(Q))} \le
\delta^\lambda \rho, $ we obtain that
\begin{multline*}
E(\D^\lambda,L_s(I^d),L_Q(I^d),
(B((S_{p,\theta}^\alpha B)^\prime(\R^d)) \cap
C_0^{\bm m}(I^d)) \mid_{I^d},
\delta^{\lambda +s^{-1} \e -q^{-1} \e} \rho) \le \\
\sup \{\| \D^\lambda f -(h_{\delta,x^0} T
h_{\delta,x^0}^{-1}) f\|_{L_q(I^d)}: f \in
(B((S_{p,\theta}^\alpha B)^\prime(\R^d)) \cap
C_0^{\bm m}(I^d)) \mid_{I^d}\}
= \\
\delta^{\lambda -q^{-1} \e}
\sup \{\| \D^\lambda (h_{\delta,x^0}^{-1} f) -
\delta^{-\lambda} T (h_{\delta,x^0}^{-1} f)\|_{L_q(Q)}: f \in (B((S_{p,\theta}^\alpha B)^\prime(\R^d)) \cap
C_0^{\bm m}(I^d)) \mid_{I^d}\}
\le \\
\delta^{\lambda -q^{-1} \e}
\sup \{\| \D^\lambda (F \mid_Q) -\delta^{-\lambda} T (F \mid_Q)\|_{L_q(Q)}: F \in (c_2 B((S_{p,\theta}^\alpha
B)^\prime(\R^d))) \cap C_0^{\bm m}(Q)\},
\end{multline*}
and, hence,
\begin{multline*} \tag{3.1.8}
c_2 \inf_{\substack{ \mathcal T \in
\mathcal B(L_s(Q),L_q(Q)):\\
\|\mathcal T\|_{\mathcal B(L_s(Q),L_q(Q))} \le \rho}}
\sup \{\| \D^\lambda (F \mid_Q) -\mathcal T (F \mid_Q)
\|_{L_q(Q)}: F \in B((S_{p,\theta}^\alpha
B)^\prime(\R^d)) \cap C_0^{\bm m}(Q)\}
\ge\\
\delta^{-\lambda +q^{-1} \e} E(\D^\lambda,L_s(I^d),
L_Q(I^d),(B(S_{p,\theta}^\alpha B)^\prime(\R^d)) \cap
C_0^{\bm m}(I^d)) \mid_{I^d},
\delta^{\lambda +s^{-1} \e -q^{-1} \e} \rho).
\end{multline*}

Denote by $ \mathcal I^D $ the linear mapping that assigns to each function $ f $ defined on a set $ D \subset \R^d $ the function $ \mathcal I^D f $ defined on $ \R^d $ by
\begin{equation*}
(\mathcal I^D f)(x) = \begin{cases} f(x), \text{ for } x \in D; \\
0, \text{ for } x \in \R^d \setminus D.
\end{cases}
\end{equation*}

Using the fact that for $ F \in C_0^{\bm m}(Q) $ the equality holds $ F \mid_D = (\mathcal I^Q(F \mid_Q)) \mid_D, $
and for $ T \in \mathcal B(L_s(D),L_q(D)) $ and $ f \in L_s(Q) $ the inequality holds
\begin{multline*}
\| (T((\mathcal I^Q f) \mid_D)) \mid_Q
\|_{L_q(Q)} \le
\| T((\mathcal I^Q f) \mid_D) \|_{L_q(D)} \le\\
\| T\|_{\mathcal B(L_s(D),L_q(D))}
\| (\mathcal I^Q f) \mid_D \|_{L_s(D)} = 
\| T\|_{\mathcal B(L_s(D),L_q(D))}
\| f\|_{L_s(Q)},
\end{multline*}
for $ T \in \mathcal B(L_s(D),L_q(D)):
\| T\|_{\mathcal B(L_s(D),L_q(D))} \le \rho, $ we derive
\begin{multline*}
\sup_{f \in B((S_{p,\theta}^\alpha B)^\prime(D))}
\| \D^\lambda f -T f\|_{L_q(D)} \ge \\
\sup_{F \in B((S_{p,\theta}^\alpha B)^\prime(\R^d)) \cap C_0^{\bm m}(Q)}
\| \D^\lambda (F \mid_D) -T (F \mid_D)
\|_{L_q(D)} \ge \\
\sup_{F \in B((S_{p,\theta}^\alpha B)^\prime(\R^d)) \cap C_0^{\bm m}(Q)}
\| (\D^\lambda (F \mid_D)) \mid_Q -(T
(F \mid_D)) \mid_Q\|_{L_q(Q)} = \\
\sup_{F \in B((S_{p,\theta}^\alpha B)^\prime(\R^d)) \cap C_0^{\bm m}(Q)}
\| \D^\lambda (F \mid_Q) -(T
((\mathcal I^Q(F \mid_Q)) \mid_D)) \mid_Q\|_{L_q(Q)} \ge \\
\inf_{ \mathcal T \in \mathcal B(L_s(Q),L_q(Q)):
\| \mathcal T\|_{\mathcal B(L_s(Q),L_q(Q))} \le \rho}
\sup_{F \in B((S_{p,\theta}^\alpha B)^\prime(\R^d)) \cap C_0^{\bm m}(Q)}
\| \D^\lambda (F \mid_Q) -\mathcal T (F \mid_Q)
\|_{L_q(Q)},
\end{multline*}
and, consequently,
\begin{multline*} \tag{3.1.9}
E(\D^\lambda,L_s(D),L_q(D),B((S_{p,\theta}^\alpha
B)^\prime(D)),\rho) \ge \\
\inf_{\substack{ \mathcal T \in \mathcal B(L_s(Q),L_q(Q)):\\
\| \mathcal T\|_{\mathcal B(L_s(Q),L_q(Q))} \le \rho}}
\sup \{\| \D^\lambda (F \mid_Q) -\mathcal T (F \mid_Q)
\|_{L_q(Q)}: F \in B((S_{p,\theta}^\alpha
B)^\prime(\R^d)) \cap C_0^{\bm m}(Q)\}.
\end{multline*}

Combining (3.1.9), (3.1.8), we conclude that for $ \rho > 0 $ the inequality (3.1.7) holds.
\end{proof}

\begin{lemma} \label{l3.1.4}

Let $ \alpha \in \R_+^d, 1 \le p,s < \infty, 1 \le q, \theta \le \infty, \bm m \in \N^d, \lambda \le \bm m $ satisfy conditions (2.1.20) and (2.2.10).
Then for each nonempty set $ J \subset
\{1,\ldots,d\} $ there exist linear mappings $ A_J: L_s((I^d)^J) \mapsto L_s(I^d) $ and $ S_J: L_q(I^d) \mapsto L_q((I^d)^J), $ possessing the following properties:

1) the inclusion holds
\begin{equation*} \tag{3.1.10}
A_J ((B((S_{p,\theta}^{\alpha^J} B)^\prime((\R^d)^J)) \cap C_0^{\bm m^J}((I^d)^J)) \mid_{(I^d)^J})
\subset
(B((S_{p,\theta}^\alpha B)^\prime(\R^d)) \cap
C_0^{\bm m}(I^d)) \mid_{I^d},
\end{equation*}

2) the inequalities hold
\begin{equation*} \tag{3.1.11}
\| A_J \|_{\mathcal B(L_s((I^d)^J),L_s(I^d)) } \le 1,
\end{equation*}
and
\begin{equation*} \tag{3.1.12}
\| S_J \|_{\mathcal B(L_q(I^d),L_q((I^d)^J))} \le 1;
\end{equation*}

3) there exists a constant $ c_5(\alpha,p,\theta,\lambda,\bm m)
> 0 $ such that for $ f \in (B((S_{p,\theta}^{\alpha^J}
B)^\prime((\R^d)^J))) \mid_{(I^d)^J} $ the equality holds
\begin{equation*} \tag{3.1.13}
S_J \D^\lambda A_J f = c_5 \D^{\lambda^J} f.
\end{equation*}
\end{lemma}

\begin{proof}
For $ J \subset \{1,\ldots,d\}: J \ne \emptyset,
\overline J = \{1,\ldots,d\} \setminus J, $ choosing some function
$$
\bm g \in B((S_{p,\theta}^{\alpha^{\overline J}}
B)^\prime((\R^d)^{\overline J})) \cap
C_0^{\bm m^{\overline J}}((I^d)^{\overline J}): \bm g \not \equiv 0,
$$
fix the functions
$$
g = (\bm g ) \mid_{(I^d)^{\overline J}} /
\max(\| \bm g \|_{L_\infty((\R^d)^{\overline J})},
\| \bm g \|_{(S_{p,\theta}^{\alpha^{\overline J}} B)^\prime((\R^d)^{\overline J})}), \\
\bm \chi = \sgn \D^{\lambda^{\overline J}} g
$$

and define the linear operators $ A_J $ and $ S_J $ by the equalities
\begin{equation*}
(A_J f)(x) = g(x^{\overline J}) f(x^J), x \in I^d, f \in L_s((I^d)^J); \\
(S_J f)(x^J) = \int_{(I^d)^{\overline J}}
\bm \chi(y^{\overline J}) f(\eta_J(x, y))
dy^{\overline J},
\end{equation*}
where $ \eta_J: \R^d \times \R^d
\mapsto \R^d $ is the mapping for which
$$
(\eta_J(x,\xi))_j = \begin{cases} x_j, j \in J; \\
\xi_j, j \in \overline J,
\end{cases}  x,\xi \in \R^d.
$$

To verify inclusion (3.1.10), it suffices to note that for
$$
\bm f \in B((S_{p,\theta}^{\alpha^J}
B)^\prime((\R^d)^J)) \cap C_0^{\bm m^J}((I^d)^J),
$$
$$
\bm g \in B((S_{p,\theta}^{\alpha^{\overline J}}
B)^\prime((\R^d)^{\overline J})) \cap
C_0^{\bm m^{\overline J}}((I^d)^{\overline J})
$$
for $ l = l(\alpha), \xi \in \R^d, \mathcal J \subset
\{1,\ldots,d\} $ the equality holds
\begin{multline*}
(\Delta_\xi^{l \chi_{\mathcal J}} (\bm g
\bm f))(x) =
((\prod_{j \in \mathcal J} \Delta_{\xi_j e_j}^{l_j})
(\bm g \bm f))(x) = \\
(((\prod_{j \in \mathcal J \cap \overline J}
\Delta_{\xi_j e_j}^{l_j}) (\prod_{j \in \mathcal J \cap J} \Delta_{\xi_j e_j}^{l_j}))
(\bm g(y^{\overline J}) \bm f(y^J)))(x) = \\
((\prod_{j \in \mathcal J \cap \overline J}
\Delta_{\xi_j e_j}^{l_j}) (\bm g(y^{\overline J})))(x^{\overline J})
((\prod_{j \in \mathcal J \cap J}
\Delta_{\xi_j e_j}^{l_j})
(\bm f(y^J)))(x^J),
\end{multline*}
and use Fubini's theorem.
It is obvious that $ \bm g(x^{\overline J}) \bm f(x^J) \in C_0^{\bm m}(I^d), $ if
$ \bm g \in C_0^{\bm m^{\overline J}}
((I^d)^{\overline J}), \bm f \in
C_0^{\bm m^J}((I^d)^J). $

Checking (3.1.11), taking into account that $ \| g \|_{L_\infty((I^d)^{\overline J})} \le 1, $ for $ f \in L_s((I^d)^J) $ we have
\begin{multline*}
\| A_J f \|_{L_s(I^d)}^s = \int_{I^d}
| g(x^{\overline J}) f(x^J)|^s dx \le
\int_{(I^d)^{\overline J} \times (I^d)^J }
| f(x^J)|^s dx^{\overline J} dx^J =\\
\int_{(I^d)^J } | f(x^J)|^s dx^J =
\| f \|_{L_s((I^d)^J)}^s,
\end{multline*}
from which (3.1.11) follows.

Next, for $ f \in L_q(I^d) $, using H\"older's inequality and Fubini's theorem, we derive
\begin{multline*}
\| S_J f \|_{L_q((I^d)^J)}^q = \int_{(I^d)^J}
| \int_{(I^d)^{\overline J}} \bm \chi(y^{\overline J})
f(\eta_J(x, y))
dy^{\overline J} |^q dx^J \le \\
\int_{(I^d)^J} (\int_{(I^d)^{\overline J}}
| \bm \chi(y^{\overline J})| \cdot
| f(\eta_J(x, y)) |
dy^{\overline J})^q dx^J \le \\
\int_{(I^d)^J} (\int_{(I^d)^{\overline J}}
| f(\eta_J(x, y)) |
dy^{\overline J})^q dx^J \le \\
\int_{(I^d)^J} \int_{(I^d)^{\overline J}}
| f(\eta_J(x, y)) |^q dy^{\overline J} dx^J = 
\int_{I^d} | f(x)|^q dx = \| f \|_{L_q(I^d)}^q,
\end{multline*}
which leads to (3.1.12).

Finally, using the definitions of the operators $A_J$, $S_J$, for $f \in (B((S_{p,\theta}^{\alpha^J} B)^\prime((\mathbb{R}^d)^J))) \mid_{(I^d)^J}$,
we obtain
\begin{multline*}
(S_J \D^\lambda A_J f )(x^J) =
\int_{(I^d)^{\overline J}}
\bm \chi(y^{\overline J})
(\D^\lambda (A_J f))(\eta_J(x, y))
dy^{\overline J} = \\
\int_{(I^d)^{\overline J}}
\bm \chi(y^{\overline J})
(\D^\lambda (g(u^{\overline J}) f(u^J)))(\eta_J(x, y))
dy^{\overline J} = \\
\int_{(I^d)^{\overline J}}
\bm \chi(y^{\overline J})
\D^{\lambda^{\overline J}} g(y^{\overline J})
\D^{\lambda^J} f(x^J) dy^{\overline J} = \\
(\int_{(I^d)^{\overline J}}
| \D^{\lambda^{\overline J}} g(y^{\overline J}) |
dy^{\overline J}) \D^{\lambda^J} f(x^J)  = 
c_5 \D^{\lambda^J} f(x^J)
\end{multline*}
with the constant $ c_5 = \int_{(I^d)^{\overline J}}
| \D^{\lambda^{\overline J}} g(y^{\overline J}) |
dy^{\overline J} > 0, $ i.e., (3.1.13) holds.
\end{proof}

\begin{propos}\label{p3.1.5}

Under the conditions of Lemma \ref{l3.1.4}, there exists a constant $ c_6(\alpha,p,\theta,\lambda,\bm m) > 0 $ such that for $ \rho > 0 $ the inequality holds
\begin{multline*} \tag{3.1.14}
E(\D^\lambda,L_s(I^d),
L_Q(I^d),(B((S_{p,\theta}^\alpha B)^\prime(\R^d)) \cap
C_0^{\bm m}(I^d)) \mid_{I^d},\rho) \ge \\
c_5 E(\D^{\lambda^J},L_s((I^d)^J),
L_Q((I^d)^J),(B((S_{p,\theta}^{\alpha^J}
B)^\prime((\R^d)^J)) \cap
C_0^{\bm m^J}((I^d)^J)) \mid_{(I^d)^J},c_6 \rho).
\end{multline*}
\end{propos}

\begin{proof}
For an operator $ T \in \mathcal B(L_s(I^d),L_q(I^d)):
\| T \|_{\mathcal B(L_s(I^d),L_q(I^d))} \le \rho, $ using the notation $ \mathfrak B = (B((S_{p,\theta}^{\alpha^J}
B)^\prime((\R^d)^J)) \cap C_0^{\bm m^J}((I^d)^J)) \mid_{(I^d)^J}, $ due to (3.1.10), (3.1.12), (3.1.13) we have
\begin{multline*}
\sup_{f \in (B((S_{p,\theta}^\alpha B)^\prime(\R^d)) \cap C_0^{\bm m}(I^d)) \mid_{I^d}}
\| \D^\lambda f -T f \|_{L_q(I^d)} \ge \\
\sup_{\mathfrak f \in (B((S_{p,\theta}^{\alpha^J}
B)^\prime((\R^d)^J)) \cap C_0^{\bm m^J}((I^d)^J)) \mid_{(I^d)^J}} \| \D^\lambda A_J \mathfrak f -
T A_J \mathfrak f \|_{L_q(I^d)} = 
\sup_{\mathfrak f \in \mathfrak B}
\| \D^\lambda A_J \mathfrak f -
T A_J \mathfrak f \|_{L_q(I^d)} = \\
\sup_{\mathfrak f \in \mathfrak B}
\| S_J \|_{\mathcal B(L_q(I^d),L_q((I^d)^J))}^{-1}
\| S_J \|_{\mathcal B(L_q(I^d),L_q((I^d)^J))}
\| \D^\lambda A_J \mathfrak f -
T A_J \mathfrak f \|_{L_q(I^d)} \ge \\
\sup_{\mathfrak f \in \mathfrak B}
\| S_J \|_{\mathcal B(L_q(I^d),L_q((I^d)^J))}
\| \D^\lambda A_J \mathfrak f -
T A_J \mathfrak f \|_{L_q(I^d)} \ge
\sup_{\mathfrak f \in \mathfrak B}
\| S_J \D^\lambda A_J \mathfrak f -
S_J T A_J \mathfrak f \|_{L_q((I^d)^J)} = \\
\sup_{\mathfrak f \in \mathfrak B}
\| c_5 \D^{\lambda^J} \mathfrak f -
S_J T A_J \mathfrak f \|_{L_q((I^d)^J)} = 
c_5 \sup_{\mathfrak f \in \mathfrak B}
\| \D^{\lambda^J} \mathfrak f -
c_6 S_J T A_J \mathfrak f \|_{L_q((I^d)^J)},
\end{multline*}
from which, taking into account that due to (3.1.11), (3.1.12) the norm
\begin{multline*}
\| S_J T A_J \|_{\mathcal B(L_s((I^d)^J),L_q((I^d)^J))} \le \\
\| S_J \|_{\mathcal B(L_q(I^d),L_q((I^d)^J))}
\| T \|_{\mathcal B(L_s(I^d),L_q(I^d))}
\| A_J \|_{\mathcal B(L_s((I^d)^J),L_s(I^d))} \le \rho,
\end{multline*}
we obtain the estimate
\begin{multline*}
\sup_{f \in (B((S_{p,\theta}^\alpha B)^\prime(\R^d)) \cap C_0^{\bm m}(I^d)) \mid_{I^d}}
\| \D^\lambda f -T f\|_{L_q(I^d)} \ge \\
c_5 \inf_{ \mathfrak T \in c_6 \rho B(\mathcal
B(L_s((I^d)^J),L_q((I^d)^J)))}
\sup_{\mathfrak f \in \mathfrak B}
\| \D^{\lambda^J} \mathfrak f - \mathfrak T \mathfrak f
\|_{L_q((I^d)^J)},
\end{multline*}
and, consequently, (3.1.14) holds.
\end{proof}

From the proof of Theorem 2.2.6 from [4], it is easy to see that the following statement holds.

\begin{propos}\label{p3.1.6}
Let $\bm{\alpha} \in \R_+, 1 \le  p < \infty, 1 \le q \le \infty, \bm \lambda \in \Z_+ $ satisfy the condition $ \bm \alpha -\bm \lambda -(1/p -1/q)_+ >0 $ and let $ 1 \le s < \infty $ and the conditions $ \bm \lambda +(1/s -1/q)_+ >0, \bm \alpha -(1/p -1/s)_+ >0 $ hold. Let also $ \bm \gamma = \bm \alpha -\bm \lambda -(1/p -1/q)_+, \bm \tau = \bm \lambda +(1/s -1/q)_+, 1 \le \theta \le \infty. $ Then there exist constants $ c_7(\bm \alpha,p,\theta,q,s,\bm \lambda) >0, \rho_1 >0 $ such that for $ \rho > \rho_1  $ the inequality holds
\begin{multline*} \tag{3.1.15}
E(\D^{\bm \lambda},L_s(I), L_Q(I),
(B((S_{p,\theta}^{\bm \alpha} B)^0(\R)) \cap
C_0^{\infty}(I)) \mid_{I},\rho) \ge c_{7}
\rho^{-\bm \gamma /\bm \tau}.
\end{multline*}
\end{propos}

\begin{theorem}\label{t3.1.7}

Under the conditions and notation of Lemma \ref{l2.2.1} and Theorem \ref{t2.2.2}, there exist constants $ c_{8}(U,X,Y,K) >0, \rho_2(U,X,Y,K) >0 $ such that for $ \rho > \rho_2 $ the inequality holds
\begin{equation*} \tag{3.1.16}
E(U,X,Y,K,\rho) \ge c_{8} \rho^{-\mn(\tau^{-1} \gamma)}.
\end{equation*}
\end{theorem}

\begin{proof}
To obtain (3.1.16), it suffices to successively apply Lemma \ref{l3.1.3}, fixing $ \bm m \ge m, $ then, choosing $ j
\in \{1,\ldots,d\}: \tau_j^{-1} \gamma_j = \mn(\tau^{-1} \gamma),
$ use Proposition \ref{p3.1.5} for $ J = \{j\}, $ and, finally, taking into account (1.1.5), apply Proposition \ref{p3.1.6} for
$ \bm \alpha = \alpha_j, \bm \lambda = \lambda_j,
\bm \gamma = \gamma_j, \bm \tau = \tau_j. $
\end{proof}

\end{document}